\documentclass[a4paper]{amsart}
\usepackage[utf8]{inputenc}
\usepackage[english]{babel}
\usepackage{amsmath, amsthm, amssymb, amsfonts}
\usepackage[foot]{amsaddr}
\usepackage{bm}
\usepackage{mathrsfs}
\usepackage{mathtools}
\usepackage{xcolor}
\usepackage{bbm}
\usepackage{graphicx}
\usepackage{caption}
\usepackage{subcaption}
\usepackage{todonotes}
\usepackage{nicefrac}
\usepackage{hyperref}
\usepackage{xspace}
\usepackage{microtype}
\usepackage{enumitem}

\hypersetup{colorlinks=true, linkcolor=blue, citecolor=blue}

\theoremstyle{plain}
\newtheorem{theorem}{Theorem}[section]
\newtheorem{corollary}[theorem]{Corollary}
\newtheorem{proposition}[theorem]{Proposition}
\newtheorem{lemma}[theorem]{Lemma}

\newtheorem{claim}{Claim}
\counterwithin*{claim}{theorem}
\theoremstyle{definition}
\newtheorem{definition}[theorem]{Definition}

\theoremstyle{remark}
\newtheorem{remark}[theorem]{Remark}

\numberwithin{equation}{section}

\newcommand{\R}{\mathbb{R}}

\newcommand{\N}{\mathbb{N}}

\newcommand{\E}{\mathbb{E}}
\newcommand{\W}{\mathbb{W}}
\newcommand{\F}{\mathscr{F}}
\newcommand{\B}{\mathscr{B}}
\renewcommand{\P}{\mathbb{P}}

\newcommand{\norm}[1]{\|#1\|}
\newcommand{\bignorm}[1]{\big\|#1\big\|}
\newcommand{\biggnorm}[1]{\bigg\|#1\bigg\|}

\newcommand{\indicator}{\mathbbm{1}}
\newcommand{\BV}{BV}
\newcommand{\meas}{\mathscr{M}}
\newcommand{\loc}{\mathrm{loc}}
\newcommand{\radon}{\mathscr{M}}

\newcommand{\initial}{\mathrm{in}}
\newcommand{\Pas}{\P\text{-a.s}}

\newcommand{\gronwconst}{E}
\newcommand{\setsep}{;\,}

\newcommand{\semimartingale}{semi\-martin\-gale\xspace}
\newcommand{\semimartingales}{semi\-martin\-gales\xspace}
\newcommand{\diffeomorphism}{diffeo\-morph\-ism\xspace}
\newcommand{\diffeomorphisms}{diffeo\-morph\-isms\xspace}
\newcommand{\componentwise}{component\-wise\xspace}

\DeclareMathOperator{\dvrg}{div}

\title[Stochastic flows for H\"older drifts]{Stochastic flows for H\"older drifts and transport/continuity equations with noise}

\email{magnusco@math.uio.no}
\date{\today}

\author{Magnus C. {\O}rke}

\begin{document}

\begin{abstract}
    We prove existence of a stochastic flow of \diffeomorphisms generated by SDEs with drift in $L^q_t C^{0, \alpha}_x$ for any $q \in [2, \infty)$ and $\alpha \in (0, 1)$. This result is achieved using a Zvonkin-type transformation for the SDE. As a key intermediate step, well-posedness and optimal regularity for a class of parabolic PDEs related to the transformation is established. Using the existence of a differentiable stochastic flow, we prove well-posedness of $\BV_\loc$-solutions of stochastic transport equations and weak solutions of stochastic continuity equations with so-called transport noise and velocity fields in $L^q_t C^{0, \alpha}_x$. For both equations, solutions may fail to be unique in the deterministic setting.
\end{abstract}

\maketitle

\section{Introduction}

Let ${(\Omega, \F, \P)}$ be a probability space and $(W_t)_{t \geq 0}$ a $d$-dimensional Brownian motion with respect to a given complete and right-continuous filtration $(\F_t)_{t \geq 0}$. For a bounded time interval $[0, T]$, we assume that $b\colon \R^d \times [0, T] \to \R^d$ is a function which satisfies the hypothesis
\begin{equation} \label{eq:h1}
    \tag{H1}
    b \in L^q\bigl((0, T); C^{0, \alpha}(\R^d; \R^d)\bigr) \qquad \text{for}\ q \in [2, \infty),\ \alpha \in (0, 1).
\end{equation}
Under these conditions, we shall be concerned with the well-posedness of the following three related equations: First, the ordinary stochastic differential equation (SDE)
\begin{equation} \label{eq:sde}
    \begin{cases}
        dX_t = b_t(X_t)\, dt + dW_t & \text{for}\ t > s,\\
        X_s = x,
    \end{cases}
\end{equation}
with $x \in \R^d$. Here, the question is not only whether the SDE is well-posed, but also if it generates a unique stochastic flow of diffeomorphisms on $\R^d$. Second, the stochastic transport equation (STE)
\begin{equation} \label{eq:stoch_transport_strong}
    d u + b \cdot \nabla u\, dt + \nabla u \circ dW_t = 0, \qquad u_{t = 0} = u_{\initial},
\end{equation}
where $u_\initial \colon \R^d \to \R$ is a function of locally bounded variation. We use the notation $b \cdot \nabla u = \sum_{i = 1}^d b^i \partial_i u$ and $\nabla u \circ dW = \sum_{i = 1}^d \partial_i u \circ dW^i_t$, where $\partial_i u \circ dW^i_t$ denotes Stratonovich integration. Thirdly, the stochastic continuity equation (SCE)
\begin{equation} \label{eq:stoch_continuity_strong}
    d \mu + \nabla \cdot (b \mu)\, dt + \nabla \mu \circ dW_t = 0, \qquad \mu_{t = 0} = \mu_\initial,
\end{equation}
where $\mu_\initial$ is a locally finite (signed) measure on $\R^d$, and $\nabla \cdot (b \mu) = \sum_{i = 1}^d \partial_i (b^i \mu)$. The ways in which we define solutions for these equations will be explained below, but let us briefly note that we seek probabilistically strong solutions, in the usual integral sense with respect to time, with weak (measure-valued) derivatives in space for the STE, and distributional derivatives in space for the SCE.

\subsection{Main results} Our main result is a well-posedness theory for the three equations \eqref{eq:sde}--\eqref{eq:stoch_continuity_strong} under hypothesis \eqref{eq:h1}. The principal tool for studying the STE and the STC in this setting is the stochastic flow generated by the SDE, provided by the following theorem.

\begin{theorem} \label{thm:main_flow}
    Assume that $b$ satisfies \eqref{eq:h1}. Then the SDE~\eqref{eq:sde} generates a stochastic flow of $C^{1, \beta}$-\diffeomorphisms $X = X_{s, t}(x, \omega)$, defined for $s, t \in [0, T]$, on~$\R^d$, for all $\beta < \alpha$. More precisely, we have the following:
    \begin{enumerate}
        \item[\textit{(i)}] For any $x \in \R^d$ and $s \in [0, T]$, the SDE \eqref{eq:sde} has a unique (strong) solution $(X_t^{s, x})_{t \in [s, T]}$. The family of solutions ${(X_t^{s, x}\setsep 0 \leq s \leq t \leq T,\, x \in \R^d)}$ has a modification, denoted by $X = X_{s, t}(x, \omega)$, which for all $\beta < \alpha$ is a forward stochastic flow of $C^{1, \beta}$-\diffeomorphisms on $\R^d$. Let $(b^{n})_{n \in \N}$ be a convergent sequence of functions such that
        \begin{equation} \label{eq:convergent_coeff}
            (b^{n}) \subset L^q_t C^{0, \alpha}_x; \qquad \lim_{n \to \infty} b^n = b\quad \text{in}\quad L^q_t C^{0, \alpha'}_x
        \end{equation}
        for some $\alpha' > 0$, with corresponding forward flows $(X^{n})_{n \in \N}$. Then for any $p > 0$,
        \begin{align}
            & \lim_{n \to \infty} \sup_{x \in \R^d} \sup_{s \in (0, T)}\E \biggl[\sup_{t \in (s, T)} |X_{s, t}^n(x) - X_{s, t}(x)|^p\biggr] = 0, \label{eq:forward_stability} \\
            & \lim_{n \to \infty} \sup_{x \in \R^d} \sup_{s \in (0, T)} \E \biggl[\sup_{t \in (s, T)} \norm{\nabla X_{s, t}^n(x) - \nabla X_{s, t}(x)}^p\biggr] = 0. \label{eq:forward_gradient_stability}
        \end{align}

        \item[\textit{(ii)}] For any $x \in \R^d$ and $t \in [0, T]$, the backward SDE
        \begin{equation} \label{eq:backward_sde}
            X_s^{x, t} = x - \int_s^t b_r(X_r^{x, t})\, dr - (W_t - W_s)
        \end{equation}
        has a unique (strong) solution $(X_s^{x, t})_{s \in [0, t]}$, and the family of solutions ${(X_s^{x, t}(x)\setsep 0 \leq s \leq t \leq T,\, x \in \R^d)}$ has a modification, denoted by ${X^{-1} = X_{s, t}^{-1}(x, \omega)}$, which is a backward stochastic flow of $C^{1, \beta}$-\diffeomorphisms on~$\R^d$ for all $\beta < \alpha$. It is $\P$-a.s.~the inverse of the forward flow $X$. If $\bigl((X^{n})^{-1}\bigr)_{n \in \N}$ are backward flows corresponding to a convergent sequence of coefficients $(b_n)_{n \in \N}$ as in \eqref{eq:convergent_coeff}, then
        \begin{align*}
            & \lim_{n \to \infty} \sup_{x \in \R^d} \sup_{t \in (0, T)} \E \biggl[\sup_{s \in (0, t)} |(X_{s, t}^{n})^{-1}(x) - X_{s, t}^{-1}(x)|^p\biggr] = 0, \\ %\label{eq:backward_stability}
            & \lim_{n \to \infty} \sup_{x \in \R^d} \sup_{t \in (0, T)} \E \biggl[\sup_{s \in (0, t)} \norm{\nabla (X_{s, t}^{n})^{-1}(x) - \nabla X_{s, t}^{-1}(x)}^p\biggr] = 0. %\label{eq:backward_gradient_stability}
        \end{align*}
        for any $p > 0$.
    \end{enumerate}
\end{theorem}

This theorem will be proved in Sec.~\ref{sec:proof_main_flow}. Relevant definitions and terminology can be found in Sec.~\ref{sec:prelims}.

The next theorem shows that the STE \eqref{eq:stoch_transport_strong} is well-posed (in the sense of Definition~\ref{def:stoch_transport_weak}), and that the solution is given by the same representation formula that one has in the smooth setting. Note that for simplicity, we write $X_t$ for $X_{0, t}$ and $X_{t}^{-1}$ for $X_{0, t}^{-1}$. 

\begin{theorem} \label{thm:main_transport}
    Assume that $b$ satisfies \eqref{eq:h1}, and let $u_\initial \in \BV_\loc(\R^d)$. Then there exists a unique $\BV_\loc$-solution of the STE \eqref{eq:stoch_transport_strong} which is given by
    \begin{equation} \label{eq:transport_formula}
        u_t(x) = u_\initial(X_t^{-1}(x)), \qquad \text{a.e.}\ x \in \R^d, \ t \in [0, T] 
    \end{equation}
    $\Pas$, where $X$ is the stochastic flow generated by \eqref{eq:sde}.
\end{theorem}
The proof of this theorem can be found in Sec.~\ref{sec:main_proofs}. We also prove an analogous result for the SCE (see Definition~\ref{def:stoch_continuity_weak}
for the definition of weak solutions):

\begin{theorem} \label{thm:main_continuity}
    Assume that $b$ satisfies \eqref{eq:h1}, and let $\mu_\initial \in \meas_\loc(\R^d)$. Then there exists a unique weak solution of the SCE \eqref{eq:stoch_continuity_strong} which is given by
    \begin{equation} \label{eq:continuity_formula}
        \mu_t = (X_t)_\# \mu_\initial,\quad \text{i.e.} \quad \int_{\R^d} \vartheta(x)\, \mu_t(dx) = \int_{\R^d} \vartheta(X_t(x))\, \mu_\initial(dx), \quad t \in [0, T]
    \end{equation}
    $\Pas$ for all $\vartheta \in C_c(\R^d)$, where $X$ is the stochastic flow generated by \eqref{eq:sde}.
\end{theorem}

\subsection{Background and approach} \label{sec:background_approach}

In the classical setting (i.e.~Lipschitz or more regular coefficients), the standard references on stochastic flows are H.~Kunita's monographs \cite{kunita, kunita_1990}, wherein linear SPDE's (stochastic partial differential equations) including the STE \eqref{eq:stoch_transport_strong} and the SCE \eqref{eq:stoch_continuity_strong} with smooth initial data are also treated.

Starting with the work of A.~K.~Zvonkin \cite{zvonkin_1974}, the conditions for strong well-posedness of nondegererate SDE's were significantly weakened by introducing a transformation of the equation which improves the regularity of its coefficients. This was later generalized to the multidimensional case by A.~J.~Veretennikov in~\cite{veretennikov_1981}. We will use a version of this approach, which can be outlined as follows. Let $v\colon \R^d \times [0, T] \to \R^d$ be a deterministic, twice continuously differentiable function. If $X_t$ is a solution of the SDE \eqref{eq:sde}, then by Itô's formula the composition $v_t(X_t)$ satisfies the system of equations
\begin{equation} \label{eq:zvonkin_ito_formula}
    dv_t(X_t) = \bigl(\partial_t v_t + b_t\cdot \nabla v_t + \frac{1}{2}\Delta v_t \bigr)(X_t)\, dt + \nabla v_t(X_t)\cdot dW_t
\end{equation}
(understood \componentwise, see \eqref{eq:forward_ito_transformation_componentwise}). If $v$ itself is a solution of the system of parabolic partial differential equations (PDEs)
\begin{equation} \label{eq:zvonkin_transform}
    \begin{cases}
    \partial_t v + b \cdot \nabla v + \frac{1}{2} \Delta v = \lambda v - b & \text{in}\ \R^d \times (0, T) \\
    v_{t = T} \equiv 0,
    \end{cases}
\end{equation}
for some $\lambda > 0$, then \eqref{eq:zvonkin_ito_formula} is equivalent to
\begin{equation*}
    dv_t(X_t) = \bigl(\lambda v_t - b_t\bigr)(X_t)\, dt + \nabla v_t(X_t)\cdot dW_t.
\end{equation*}
Let $g_t(x) \coloneqq x + v_t(x)$, and assume that $g_t(\cdot)$ is invertible on $\R^d$. Then the process $Y_t \coloneqq g_t(X_t)$ satisfies
\begin{equation} \label{eq:transformed_system}
    dY_t = \lambda v_t(g_t^{-1}(Y_t))\, dt + \bigl(I + \nabla v_t(g_t^{-1}(Y_t))\bigr)dW_t,
\end{equation}
where $I$ is the identity matrix. The upshot here is that the coefficients in \eqref{eq:transformed_system} turn out to be more regular than the coefficients in the original SDE, provided that the solution $v$ of \eqref{eq:zvonkin_transform} is regular enough.

In \cite{flandoli_gubinelli_priola}, F.~Flandoli, M.~Gubinelli and E.~Priola used the transformation sketched above (hereafter referred to as a Zvonkin-type transformation) to show that SDEs with drift $b \in L^\infty((0, T); C^{0, \alpha}(\R^d; \R^d))$ and constant diffusion generate unique forward stochastic flows of $C^{1, \beta}$-\diffeomorphisms, for all $\beta < \alpha$. Based on the existence and regularity of such a flow, they proved that the STE with initial data $u_\initial \in L^\infty(\R^d)$ has a unique weak solution given by the representation formula \eqref{eq:transport_formula}, provided additionally that the drift satisfies $\dvrg b \in L^p(\R^d \times [0, T])$ for some $p \in (2, \infty)$, or alternatively $\dvrg b \in L^1_\loc(\R^d \times [0, T])$ and $\alpha > 1/2$. Moreover, they proved that if $u_\initial \in \BV_\loc$, then $\BV_\loc$-solutions exist and are unique without the condition on the divergence of~$b$. It is clear that the results and methods from \cite{flandoli_gubinelli_priola} have served as an inspiration for the present study.

Similar ideas have been applied elsewhere. In \cite{wei_duan_gao_lv_2021}, J.~Wei, J.~Duan, H.~Gao and G.~Lv proved that there exists a forward stochastic stochastic flow of \diffeomorphisms generated by the SDE \eqref{eq:sde} if the drift $b$ satisfies hypothesis \eqref{eq:h1} with $\alpha > 2/q$. Although this is similar to Theorem~\ref{thm:main_flow}, note that the additional requirement $\alpha > 2/q$---which here will be removed---is quite restrictive. Moreover, we will have to deal with both the forward and the backward flow in order to prove uniqueness for both the STE and the SCE. E.~Fedrizzi and F.~Flandoli showed in \cite{fedrizzi_flandoli_2013} that if
\begin{equation}\label{eq:kr_condition}
    b \in L^q((0, T); L^p(\R^d; \R^d)) \qquad \text{for}\ \frac{2}{q} + \frac{d}{p} < 1, \quad p \geq 2,\ q > 2,
\end{equation}
then the SDE \eqref{eq:sde} generates a forward stochastic flow of homeomorphisms. Although there is a flaw in their proof (see the comment in Appendix \ref{appendix:gronwall}), the result seems to be correct. We also mention the paper \cite{krylov_rockner_2005} by N.~V.~Krylov and M.~R\"ockner, where strong local existence and uniqueness of the SDE \eqref{eq:sde} is proved under condition~\eqref{eq:kr_condition}. This result clearly covers \eqref{eq:sde} given \eqref{eq:h1}, but does not provide a differentiable stochastic flow.

Other related works on stochastic flows include \cite{beck_flandoli_gubinelli_maurelli_2019, figalli_2008,zhang_2010}, which develop the theory of stochastic Lagrangian flows in the vein of the deterministic DiPerna--Lions theory of regular Lagrangian flows, and \cite{attanasio_2010}, which addresses a one-dimensional autonomous drift of bounded variation with bounded distributional derivative.

As outlined above, the analysis of the SDE \eqref{eq:sde} proceeds via the PDE \eqref{eq:zvonkin_transform}, for which we require results on existence and regularity. Parabolic equations in $L^q_t C^{2, \alpha}_x$-spaces, albeit without lower-order terms, have been studied by N.~V.~Krylov in  \cite{krylov_2002}, where a priori estimates were established. In our setting, in addition to existence and regularity, it is necessary to show that $\norm{\nabla v^i}_{L^\infty(\R^d)}$ can be made arbitrarily small at the expense of the parameter $\lambda$, ensuring that the function $x \mapsto x + v_t(x)$ is invertible. To this end, we will provide a self-contained proof tailored to our requirements.

We will consider $\BV_\loc$-solutions to the STE and weak (measure-valued) solutions to the SCE. These classes of solutions are quite natural in this context, as the two equations are in some sense dual (cf.~Lemma \ref{lemma:duality}), and under hypothesis \eqref{eq:h1}, well-posedness is new. The closest result is, as already mentioned, the existence and uniqueness of $\BV_\loc$-solutions for the STE when $b \in L^\infty_t C^{0, \alpha}_x$, proved in \cite{flandoli_gubinelli_priola}. We have not been able to find comparable results for the SCE in the existing literature.

Since we do not consider distributional solutions to the STE (which would require taking the divergence of $b$ in the weak formulation), we will here not depend on conditions on $\dvrg b$. In \cite{wei_duan_gao_lv_2021}, existence and uniqueness of distributional solutions was proved given that $\dvrg b \in L^1_tL^\infty_x$ and Sobolev initial data, in addition to \eqref{eq:h1} with ${\alpha > 2/q}$. A similar result has been derived by W.~Neves and C.~Olivera in \cite{neves_olivera_2015} under condition \eqref{eq:kr_condition} and zero divergence on the drift. In the extensive paper \cite{beck_flandoli_gubinelli_maurelli_2019}, L.~Beck, F.~Flandoli, M.~Gubinelli and M.~Maurelli proved well-posedness results for the STE and SCE with drifts satisfying Ladyzhenskaya--Prodi--Serrin-type conditions (similar to \eqref{eq:kr_condition}) with Sobolev initial data. We also mention the paper \cite{gess_smith_2019} by B.~Gess and S.~Smith, dealing with a stochastic continuity equation with a nonlinear multiplicative noise.

\subsection{Outline}

The paper is outlined as follows. Sec.~\ref{sec:prelims} collects relevant preliminaries on function spaces and the stochastic formalism. In Sec.~\ref{sec:pde}, we focus entirely on the analysis of the deterministic PDE \ref{eq:zvonkin_transform}, proving existence, uniqueness and regularity of solutions. This is then used in Sec.~\ref{sec:sde} to analyze the SDE \eqref{eq:sde}, culminating in the proof of Theorem~\ref{thm:main_flow} in Sec.~\ref{sec:proof_main_flow}. Finally, in Sec.~\ref{sec:main_proofs}, we leverage the existence and regularity of the stochastic flow to prove well-posedness for the SPDEs \eqref{eq:stoch_transport_strong}--\eqref{eq:stoch_continuity_strong}.

\subsection{Notation}

Throughout, $|\cdot|$ is the Euclidian norm on $\R^d$, and $\norm{\cdot}$ is the Hilbert-Schmidt matrix norm. Temporal evaluation will often be written in subscript (e.g.~$v_t$, $W_t$), while spatial indices will be written in superscript (e.g.~$x^i$, $v^i$). An exception is $\partial_i f$, by which we denote the partial derivative of $f$ with respect to $x^i$. A standard mollifier $\rho^\varepsilon$ on $\R^d$ is a function ${\rho^\varepsilon(x) = \varepsilon^{-d} \rho(x/\varepsilon)}$ for $0 \leq \rho \in C^\infty_c(\R^d)$ with $\int_{\R^d} \rho\, dx = 1$. We denote by ${a \vee b}$ the maximum of $a$ and $b$. When dealing with a constant $C$, the expression ${C = C(a)}$ means that the constant depends only on the quantity $a$. We also use the notation~$\lesssim_a$, meaning that an inequality holds up to a constant $C(a)$.

\section{Preliminaries} \label{sec:prelims}

\subsection{Function spaces}

The space $C^{0, \alpha}(\R^d)$, for $\alpha \in (0, 1)$, will be the collection of bounded functions $f\colon\R^d \to \R$ with norm
\begin{equation*}
    \norm{f}_{C^{0, \alpha}(\R^d)} = \norm{f}_{C^0(\R^d)} + [f]_{C^{0, \alpha}(\R^d)} = \sup_{x \in \R^d} |f(x)| + \sup_{\substack{x, y \in \R^d \\ x \neq y}} \frac{|f(x) - f(y)|}{|x-y|^\alpha} < \infty.
\end{equation*}
(If $f$ is a matrix-valued function, replace the Euclidian norm by the matrix norm.) For $k \in \N_0 \coloneqq \{0, 1, 2, ...\}$, the space $C^{k, \alpha}(\R^d)$ comprises functions with finite norm
\begin{equation*}
    \norm{f}_{C^{k, \alpha}(\R^d)} = \norm{f}_{C^k(\R^d)} + [f]_{C^{k, \alpha}(\R^d)} = \sum_{j = 0}^k \max_{|\nu| = j} \norm{\partial_\nu f}_{C^0(\R^d)} + \max_{|\nu| = k} [\partial_\nu f]_{C^{0, \alpha}(\R^d)},
\end{equation*}
where $\nu$ is any multi-index $\nu = (\nu_1, ..., \nu_d)$. Let $m \in \N_0$ such that $m \leq k$. Then an equivalent norm for $C^{k, \alpha}(\R^d)$ is, abusing notation,
\begin{equation} \label{eq:zygmund_norm}
    \norm{f}_{C^{k, \alpha}(\R^d)} = \norm{f}_{C^m(\R^d)} + \max_{|\nu| = m}\ \sup_{\substack{h \in \R^d \\ |h| \leq 1}} \frac{\bignorm{\Delta_h^{k+1-m}[\partial_\nu f]}_{C^0(\R^d)}}{|h|^{k+\alpha-m}},
\end{equation}
where $\Delta_h^l[\cdot]$ denotes the $l$-th order iterated difference operator ${\Delta_{h}[f]}$ given by ${\Delta_{h}[f](x) = f(x+h) - f(x)}$ (see \cite[Sec.~1.2.2]{triebel_1992}). Let $L^q((0, T); C^{k, \alpha}(\R^d))$ denote the closure of $C^\infty_c((0, T); C^{k, \alpha}(\R^d))$ in the norm
\begin{equation*}
    \norm{f}_{L^q((0, T); C^{k, \alpha}(\R^d))} = \biggl(\int_0^T \norm{f_t}_{C^{k, \alpha}(\R^d)}^q\, dt\biggr)^{\frac{1}{q}}.
\end{equation*}
Note that while $C^{k, \alpha}(\R^d)$ is not a separable space, strong measurability is implicitly assumed in the above construction of $L^q((0, T); C^{k, \alpha}(\R^d))$ (see \cite{krylov_cz_2002} for more information about these spaces). For simplicity, $L^q(0, T)$ will be abbreviated as $L^q_t$ and similarly $L^q((0, T), C^{k, \alpha}(\R^d))$ as $L^q_t C^{k, \alpha}_x$.

Let $\meas_\loc(\R^d)$ be the space of locally finite measures on $\R^d$, meaning that ${\mu \in \meas_\loc(\R^d)}$ if $|\mu|(K) < \infty$ for all compact $K \subset \R^d$, where $|\mu|$ denotes the total variation measure of $\mu$. A function ${f \in L^1_{\loc}(\R^d)}$ belongs to $\BV_\loc(\R^d)$ if its distributional derivatives can be represented as measures $\partial_i f \in \radon_{\loc}(\R^d)$, for $i = 1, ..., d$. We refer to \cite{ambrosio_fusco_pallara_2000} for details on the spaces of $\BV_\loc$-functions and locally finite measures.

\subsection{Forward and backward Itô integral}

Recall that $(\Omega, \F, (\F_t)_t, \P)$ is a given stochastic basis with a $d$-dimensional Brownian motion $(W_t)_t$. Here, there will be no loss in generality in assuming that $(\F_t)_t$ is the filtration generated by $(W_t)_t$. A two-parameter filtration $(\F_{s, t})_{0 \leq s \leq t \leq T}$ is a complete family of sub $\sigma$-algebras of $\F$ which satisfies $\F_{s, t} \subset \F_{s', t'}$ for all $s' \leq s \leq t \leq t'$ and moreover ${\cap_{\varepsilon > 0} \F_{s, t+\varepsilon} = \F_{s, t}}$ and ${\cap_{\varepsilon > 0} \F_{s-\varepsilon, t} = \F_{s, t}}$. We will take $(\F_{s, t})_{0 \leq s \leq t \leq T}$ to be the completed $\sigma$-algebra generated by $(W_u - W_r)_{s \leq r \leq u \leq t}$. Let $M_t$ be a continuous stochastic process on~$[0, T]$. It is a forward martingale adapted to $\F_{s, t}$ if it is integrable, the process $M_t - M_s$ is $\F_{s, t}$-measurable and satisfies $\E[M_r - M_s| \F_{s, t}] = M_t - M_s$ for all $s \leq t \leq r$. Similarly, it is a backward martingale if $\E[M_r - M_t| \F_{s, t}] = M_s - M_t$ for all $r \leq s \leq t$. Let $X_t$ be a continuous stochastic process which is $\F_{s, t}$-adapted. It is a forward/backward \semimartingale if it is a sum of a forward/backward $\F_{s, t}$-martingale and a continuous $\F_{s, t}$-adapted process of bounded variation.

Let $X_t$ be a backward \semimartingale with respect to $\F_{s, t}$. Then the backward Itô integral is defined as
\begin{equation*}
    \int_s^t f_r\, \hat{d}X_r \coloneqq \lim_{n \to \infty} \sum_{k = 0}^{n-1} f_{t_{k+1} \vee s}(X_{t_{k+1} \vee s} - X_{t_{k} \vee s}),
\end{equation*}
whenever the right-hand side converges in probability for any sequence of partitions ${0 = t_0 \leq t_1 \leq ... \leq t_n = t}$ of $[0, t]$ with vanishing mesh size. The backward stochastic calculus in this setting is completely parallell to the forward calculus; see~\cite{kunita_1990}.

Let $X_t(x)$ be a (forward) stochastic process with a parameter $x \in \R^d$. For $k \in \N_0$, it is called a $C^{k, \beta}$-valued process if for all $t \in [0, T]$ and $\P$-a.e.~$\omega \in \Omega$, the map $x \mapsto X_t(x, \omega)$ belongs to $C^{k, \beta}(\R^d)$. If in addition $\partial_\nu X_t(x)$ is a family of continuous \semimartingales for all $|\nu| \leq k$, then is is called a (forward) \mbox{$C^{k, \beta}$-\semimartingale}.

\subsection{Stochastic flows}

Following H.~Kunita, we begin by defining stochastic flows without reference to SDEs.

\begin{definition}[Kunita \cite{kunita_1990}] \label{def:stoch_flow}
    Let $X = X_{s, t}(x, \omega)$, defined for $s, t \in [0, T]$ and $x \in \R^d$, be a continuous $\R^d$-valued random field on $(\Omega, \F, \P)$. It is a stochastic flow of homeomorphisms if for $\P$-a.e.~$\omega \in \Omega$, the family $(X_{s, t}(\omega))_{s, t \in [0, T]}$ is a flow of homeomorphisms, i.e.
    \begin{itemize}
        \item[(i)] $X_{s, t}(\omega) = X_{r, t}(\omega) \circ X_{s, r}(\omega)$ for all $s,r, t \in [0, T]$,
        \item[(ii)] $X_{s, s}(\omega, x) = x$ for all $s \in [0, T]$ and $x \in \R^d$,
        \item[(iii)] $X_{s, t}(\omega)\colon \R^d \to \R^d$ is a homeomorphism on $\R^d$ for all $s, t \in [0, T]$.
    \end{itemize}
    We say that $X$ is a stochastic flow of $C^k$-\diffeomorphisms if
    \begin{itemize}
        \item[(iv)] $X_{s, t}(\omega)$ is $k$ times differentiable with respect to $x$, and the derivatives are continuous in $(x, s, t)$.
    \end{itemize}
    If in addition to (iv), the derivatives are H\"older continuous with exponent $\beta$ with respect to $x$, we say that $X$ is a stochastic flow of $C^{k, \beta}$-\diffeomorphisms.
\end{definition}

The restriction of the flow to the forward temporal variables $(X_{s, t})_{0 \leq s \leq t \leq T}$ will be called the forward flow. Similarly, the backward flow is the restriction to the backward variables. It will be denoted by $(X_{s, t}^{-1})_{0 \leq s \leq t \leq T}$, in view of the relation $X_{s, t}^{-1} = X_{t, s}$ for $0 \leq s \leq t \leq T$.

Stochastic flows of homeomorphisms can be related to solutions of stochastic differential equations as follows. Let $(X_t^{x, s}\setsep 0 \leq s \leq t \leq T,\, x \in \R^d)$ be a family of solutions to the SDE \eqref{eq:sde}. We shall say that a forward stochastic flow of homeomorphisms is generated by the SDE \eqref{eq:sde} if it is a modification of the system $(X_t^{x, s}\setsep 0 \leq s \leq t \leq T,\, x \in \R^d)$ of solutions. Of course, this also implies that for any $x \in \R^d$ and $s \in [0, T]$, the process $(X_{s, t}(x))_{t \in [s, T]}$ (or a version of it) is a solution of \eqref{eq:sde}. Equivalently, a backward flow of homeomorphisms $X_{s, t}^{-1}(x)$ is said to be generated by \eqref{eq:backward_sde} if it is a modification of the system of backward solutions $(X_s^{x, t}\colon 0 \leq s \leq t \leq T,\, x \in \R^d)$.

\section{Parabolic PDEs with coefficients in \texorpdfstring{$L^q_t C^{0, \alpha}_x$}{TEXT}} \label{sec:pde}

In this section, we consider the PDE
\begin{equation} \label{eq:model_eq_classical}
    \begin{cases}
        \partial_t v + b \cdot \nabla v + \lambda v = \Delta v + f & \text{in}\ \R^d \times (0, T), \\
        v_{t = 0} \equiv 0,
    \end{cases}
\end{equation}
where $v\colon \R^d \times [0, T] \to \R$. This is a model equation for the components of \eqref{eq:zvonkin_transform} (one can easily convert to a terminal value problem by a change of variables). Given the formal calculations in Sec.~\ref{sec:background_approach}, it makes sense assume that $b$ satisfies \eqref{eq:h1} and in addition
\begin{equation} \label{eq:h2}
    \tag{H2}
    f \in L^q((0, T); C^{0, \alpha}(\R^d)) \qquad \text{for}\ q \in [2, \infty),\ \alpha \in (0, 1).
\end{equation}
Our aim is to find solutions of \eqref{eq:model_eq_classical} in the class
\begin{equation*}
    \W^{1, q}_{2, \alpha}(T) \coloneqq \bigl\{v \in L^{q}_t C^{2,\alpha}_{x} \cap L^\infty_t C^{1, \alpha}_{x},\ \partial_t v \in L^{q}_t C^{0, \alpha}_{x}\bigr\},
\end{equation*}
with bounded norm $\norm{v}_{\W^{1, q}_{2, \alpha}(T)} \coloneqq \norm{v}_{L^{q}_t C^{2,\alpha}_{x}} + \norm{v}_{L^\infty_t C^{1, \alpha}_{x}} + \norm{\partial_t v}_{L^{q}_t C^{0, \alpha}_{x}}$, according to the following definition.

\begin{definition} \label{def:model_eq_weak}
    A solution of \eqref{eq:model_eq_classical} is a function $v \in \W^{1, q}_{2, \alpha}(T)$ which satisfies
    \begin{equation} \label{eq:model_eq_weak}
        \int_0^T \int_{\R^d} v \partial_t \varphi\, dx dt = \int_0^T \int_{\R^d} \bigl(b \cdot \nabla v + \lambda v - \Delta v - f\bigr)\varphi\, dx dt
    \end{equation}
    for all $\varphi \in C^\infty_c(\R^d \times [0, T))$.
\end{definition}

We begin with proving the familiar mild form/Duhamel representation formula for solutions, from which we derive uniqueness and a priori regularity.

\subsection{Uniqueness and a priori regularity}

The calculations in this section will depend heavily on the explicit formula for the fundamental solution of the \mbox{$d$-dimensional} heat equation, which we denote by $K_t(x) \coloneqq (4 \pi t)^{-\nicefrac{d}{2}}\, \exp(-\nicefrac{|x|^2}{4 t})$.

\begin{lemma}
    Assume that $b$ and $f$ satisfy \eqref{eq:h1}--\eqref{eq:h2}. If $v$ is a solution of \eqref{eq:model_eq_classical}, then it satisfies
    \begin{equation} \label{eq:zvonkin_duhamel}
        v_t(x) = \int_0^t e^{-\lambda(t-s)} \bigl(K_{t-s} * (f_s - b_s\cdot\nabla v_s)\bigr)(x)\, ds
    \end{equation}
    for all $x \in \R^d$ and a.e.~$t \in (0, T)$.
\end{lemma}

\begin{remark}
    With nonzero initial data $v_{t = 0} = v_\initial$, the solution would be given by the implicit formula
    \begin{equation*} \label{eq:zvonkin_duhamel_nonzero_initial}
        v_t(x) = e^{-\lambda t} (K_t * v_\initial)(x) + \int_0^t e^{-\lambda(t-s)} \bigl(K_{t-s} * (f_s - b_s\cdot\nabla v_s)\bigr)(x)\, ds.
    \end{equation*}
    A natural class of initial data in this setting would be $C^{0, \alpha}(\R^d)$, for which solutions in $\W^{2, \alpha}_{1, q}(T)$ can be obtained.
\end{remark}

\begin{proof}
    For arbitrary $\psi \in C^\infty_c(\R^d \times [0, T])$, define the function
    \begin{equation*}
        \varphi(x, t) \coloneqq -\int_t^T e^{-\lambda(s-t)} \bigl(K_{s-t}(\cdot) * \psi(\cdot, s)\bigr)(x)\, ds.
    \end{equation*}
    Then $\varphi$ satisfies the equation $\partial_t \varphi + \Delta \varphi = \lambda \varphi + \psi$ on $\R^d \times (0, T)$, with terminal condition $\varphi(x, T) \equiv 0$. Moreover, due to the exponential decay of the heat kernel, it is easy to see by an approximation argument that $\varphi$ can be used as a test function for $v$ in \eqref{eq:model_eq_weak}. This gives
    \begin{equation*}
        \begin{aligned}
            \int_0^T \int_{\R^d} v \psi\, dx dt 
            & = \int_0^T \int_{\R^d} v(\partial_t \varphi + \Delta \varphi - \lambda \varphi)\, dx dt = \int_0^T \int_{\R^d} (b \cdot \nabla v - f) \varphi\, dx dt \\
            & = -\int_0^T \int_{\R^d} \int_t^T (b \cdot \nabla v - f) e^{-\lambda(s-t)} \bigl(K_{s-t}(\cdot) * \psi(\cdot, s)\bigr)(x)\, ds dx dt \\
            & = \int_0^T \int_{\R^d} \psi \biggl(\int_0^t e^{-\lambda (t-s)} (K_{t-s}*(f_s - b_s \cdot \nabla v_s)) \, ds\biggr)\, dx dt
        \end{aligned}
    \end{equation*}
    where we have changed the order of integration in the last equality by Fubini's theorem. Since $\psi$ was arbitrary, this proves \eqref{eq:zvonkin_duhamel}.
\end{proof}

The next lemma, which is well known (e.g.~in the theory of Schauder estimates), will be used repeatedly below.

\begin{lemma} \label{lemma:heat_kernel_diff}
    Let $f \in C^{0, \alpha}(\R^d)$ for some $\alpha \in (0, 1)$. Then there exists a constant $C = C(\alpha, d) > 0$ such that
    \begin{equation} \label{eq:convolution_estimate}
        \norm{\partial_\nu (K_t * f)}_{L^\infty_x} \leq C t^{\frac{\alpha - |\nu|}{2}} [f]_{C^{0, \alpha}_x}
    \end{equation}
    for all $t > 0$ and $|\nu| \in \{1, 2, 3\}$.
\end{lemma}

\begin{proof}
    Since the derivatives $\partial_\nu K_t(x)$ integrate to zero on $\R^d$ for any $t > 0$ and $|\nu| \in \{1, 2, 3\}$, we have
    \begin{equation*} \label{eq:heat_kernel_symmetry}
        \begin{aligned}
            |\partial_\nu(K_t * f)(x)| = \bigl|\int_{\R^d} \partial_\nu K_{t}(x-y) f(y)\, dy\bigr| & = \bigl|\int_{\R^d} \partial_\nu K_{t}(x-y) (f(y) - f(x))\, dy\bigr| \\
            & \leq [f]_{C^{0, \alpha}_x} \int_{\R^d} |\partial_\nu K_{t}(x-y)| |x-y|^\alpha\, dy.
        \end{aligned}
    \end{equation*}
    Using that $|\partial_\nu K_{t}(x)| \lesssim t^{-\nicefrac{|\nu|}{2}} K_{2t}(x)$ on $\R^d$, the last integral can be bounded by
    \begin{equation*}
        \begin{aligned}
            \int_{\R^d} |\partial_\nu K_{t}(x-y)| |x-y|^\alpha\, dy & \lesssim \frac{1}{t^\frac{|\nu|}{2}} \frac{1}{(4 \pi t)^{\frac{d}{2}}} \int_{\R^d} e^{-\frac{|x-y|^2}{8 t}} |x-y|^\alpha\, dy \\
            & = \frac{1}{t^\frac{|\nu|-\alpha}{2}} \frac{1}{(4\pi)^\frac{d}{2}} \int_{\R^d} e^{-\frac{|z|^2}{8}} |z|^\alpha\, dz,
        \end{aligned}
    \end{equation*}
    where we have used the change of variables $z = (x-y)/\sqrt{t}$. Since the final integral is bounded by some constant depending only on $\alpha$ and $d$, we conclude that \eqref{eq:convolution_estimate} holds.
\end{proof}

\begin{proposition}[A priori regularity] \label{prop:apriori_regularity}
    Assume that $b$ and $f$ satisfy \eqref{eq:h1}--\eqref{eq:h2}. If $v$ is a solution of \eqref{eq:model_eq_classical}, then for any $\lambda > 0$ there exists a constant
    \begin{equation*} \label{eq:a_priori_constant}
        C_{f} = C_{f}\bigl(\alpha, d, q, T, \norm{b}_{L^q_t C^{0, \alpha}_{x}}, \norm{f}_{L^q_t C^{0, \alpha}_{x}}\bigr) \geq 0
    \end{equation*}
    such that $\norm{v}_{\W^{1, q}_{2, \alpha}} \leq C_{f}$. The constant $C_f$ is continuous with respect to $\norm{f}_{L^q_t C^{0, \alpha}_{x}}$ and admits the limit
    \begin{equation} \label{eq:apriori_limit_f}
        \lim_{\norm{f}_{L^q_t C^{0, \alpha}_{x}} \to 0}\, C_{f} = 0.
    \end{equation}
    In particular, solutions are unique. Furthermore, we have $\lim_{\lambda \to \infty} \norm{v}_{L^{\infty}_t C^{1}_{x}} = 0$.
\end{proposition}

\begin{proof}
    Let $v$ be a solution of \eqref{eq:model_eq_classical}. We first make the following claim.

    \begin{claim} \label{claim:1}
        There is a constant $C_{f}^{(1)} = C_{f}^{(1)} \bigl(\alpha, d, q, T, \norm{b}_{L^q_t C^{0, \alpha}_{x}}, \norm{f}_{L^q_t C^{0, \alpha}_{x}}\bigr) \geq 0$ such that $\norm{v}_{L^\infty_t C^{1, \alpha}_{x}} \leq C_{f}^{(1)}$.
    \end{claim}

    \begin{proof} [Proof of Claim \ref{claim:1}]
        To estimate $\norm{v}_{L^\infty_t C^{1, \alpha}_{x}}$, it would suffice to bound
        \begin{equation*}
            \eta_{h}(t) \coloneqq \norm{v_t}_{L^\infty_x} + \norm{\nabla v_t}_{L^\infty_x} + \frac{\norm{\Delta_h[\nabla v_t]}_{L^\infty_x}}{|h|^\alpha}
        \end{equation*}
        uniformly for a.e.~$t \in (0, T)$ and $h \in \R^d \setminus \{0\}$. Using Young's convolution inequality and that $\norm{K_{t-s}}_{L^1_x} \lesssim_{d} 1$ yields
        \begin{equation} \label{eq:v_bound}
            \begin{aligned}
                \norm{v_t}_{L^\infty_x} & \lesssim_{d} \int_0^t e^{-\lambda(t-s)} \bigl(\norm{f_s}_{L^\infty_x} + \norm{b_s}_{L^\infty_x} \norm{\nabla v_s}_{L^\infty_x}\bigr)\, ds \\
                & \leq \int_0^t \norm{f_s}_{L^\infty_x} + \norm{b_s}_{L^\infty_x} \norm{\nabla v_s}_{L^\infty_x}\, ds
            \end{aligned}
        \end{equation}
        Differentiating \eqref{eq:zvonkin_duhamel}, we obtain
        \begin{equation} \label{eq:dv_bound}
        \begin{aligned}
            \norm{\nabla v_t}_{L^\infty_x} & \lesssim_{\alpha, d} \int_0^t \frac{e^{-\lambda(t-s)}}{(t-s)^\frac{1-\alpha}{2}} [f_s - b_s \cdot \nabla v_s]_{C^{0, \alpha}_x}\, ds \\
            & \lesssim_d \int_0^t \frac{1}{(t-s)^\frac{1-\alpha}{2}} \bigl(\norm{f_s}_{C^{0, \alpha}_{x}} + \norm{b_s}_{C^{0, \alpha}_{x}} \norm{v_s}_{C^{1, \alpha}_{x}}\bigr)\, ds,
        \end{aligned}
        \end{equation}
        where we have used Lemma \ref{lemma:heat_kernel_diff}. To estimate the term $\norm{\Delta_h[\nabla v_t]}_{L^\infty_x} / |h|^\alpha$ in $\eta_h$, observe first that if $|h| > 1$, then
        \begin{equation} \label{eq:large_h}
            \frac{\norm{\Delta_h[\nabla v_t]}_{L^\infty_x}}{|h|^\alpha} < 2 \norm{\nabla v_t}_{L^\infty_x}.
        \end{equation}
        If on the other hand $|h| \leq 1$, we have
        \begin{equation} \label{eq:first_diff}
            \norm{\Delta_h[\nabla v_t]}_{L^\infty_x} \leq \int_0^t \bignorm{\Delta_h\bigl[\nabla\bigl(K_{t-s} * (f_s - b_s \cdot \nabla v_s)\bigr)\bigr]}_{L^\infty_x}\, ds,
        \end{equation}
        which can be obtained by differentiating \eqref{eq:zvonkin_duhamel} and applying the difference operator~$\Delta_h$. We will use two different ways to bound the integrand on the right-hand side of \eqref{eq:first_diff}: one on the interval $\bigl(0, (t-|h|^2) \vee 0\bigr)$, and one on the remaining interval $\bigl((t-|h|^2) \vee 0, t\bigr)$. First, by Lemma \ref{lemma:heat_kernel_diff}, we have for any $g \in C^{0, \alpha}(\R^d)$ and $i = 1,..., d$ that
        \begin{equation*}
            \begin{aligned}
                \bignorm{\Delta_h[\partial_i (K_{t-s} * g)]}_{L^\infty_x} & = \biggnorm{\int_0^1 \frac{\partial}{\partial r} \partial_i (K_{t-s} * g)(\cdot + rh)\, dr}_{L^\infty_x} \\
                & \leq |h| \bignorm{\nabla(\partial_i(K_{t-s} * g))}_{L^\infty_x} \lesssim_{\alpha, d} \frac{|h|}{(t-s)^\frac{2-\alpha}{2}} [g]_{C^{0, \alpha}_x}.
            \end{aligned}
        \end{equation*}
        Second, we use
        \begin{equation*}
            \bignorm{\Delta_h [\nabla (K_{t-s} * g)]}_{L^\infty_x} \leq 2 \bignorm{\nabla (K_{t-s} * g)}_{L^\infty_x} \lesssim_{\alpha, d} \frac{1}{(t-s)^\frac{1-\alpha}{2}} [g]_{C^{0, \alpha}_x},
        \end{equation*}
        again by Lemma \ref{lemma:heat_kernel_diff}. Inserted in \eqref{eq:first_diff}, these yield
        \begin{equation*}
            \begin{aligned}
                \norm{\Delta_h[\nabla v_t]}_{L^\infty_x} & \lesssim_{\alpha, d} |h| \int_0^{(t-|h|^2) \vee 0} \frac{1}{(t-s)^{\frac{2-\alpha}{2}}} [f_s - b_s \cdot \nabla v_s]_{C^{0, \alpha}_x}\, ds \\
                & \quad + \int_{(t-|h|^2) \vee 0}^t \frac{1}{(t-s)^{\frac{1-\alpha}{2}}} [f_s - b_s \cdot \nabla v_s]_{C^{0, \alpha}_x}\, ds.
            \end{aligned}
        \end{equation*}
        Dividing by $|h|^{\alpha}$, we have
        \begin{equation} \label{eq:dv_Holder}
        \begin{aligned}
            \frac{\norm{\Delta_h[\nabla v_t]}_{L^\infty_x}}{|h|^\alpha} \lesssim_{\alpha, d} \int_0^t \biggl(& \frac{|h|^{1-\alpha}}{(t-s)^{\frac{2-\alpha}{2}}} \indicator_{\{|h|^2 < t-s\}} + \frac{|h|^{-\alpha}}{(t-s)^{\frac{1-\alpha}{2}}} \indicator_{\{t - s < |h|^2\}}\biggr) \\
            & \qquad \qquad \times \bigl(\norm{f_s}_{C^{0, \alpha}_{x}} + \norm{b_s}_{C^{0, \alpha}_{x}} \norm{v_s}_{C^{1, \alpha}_{x}}\bigr)\, ds.
        \end{aligned}
        \end{equation}
        For $s \in (0, t)$, define the function
        \begin{equation*} \label{eq:w_definition}
        \begin{aligned}
            & w_h(t - s) \coloneqq 1 + \frac{1}{(t-s)^{\frac{1-\alpha}{2}}} \\
            & \quad + \biggl(\frac{|h|^{1-\alpha}}{(t-s)^{\frac{2-\alpha}{2}}} \indicator_{\{|h|^2 < t-s\}} + \frac{|h|^{-\alpha}}{(t-s)^{\frac{1-\alpha}{2}}} \indicator_{\{t - s < |h|^2\}}\biggr)\indicator_{\{|h| \leq 1\}}.
        \end{aligned}
        \end{equation*}
        The motivation for this definition is to organize the coefficients that appear in the estimates for each term in $\eta_h(t)$ in a single function. Indeed, adding together \eqref{eq:v_bound}, \eqref{eq:dv_bound}, \eqref{eq:large_h} and \eqref{eq:dv_Holder}, we see that
        \begin{equation} \label{eq:pre_gronwall}
            \eta_{h}(t) \lesssim_{\alpha, d} \int_0^t w_h(t - s) \norm{f_s}_{C^{0, \alpha}_{x}}\, ds + \int_0^t w_h(t - s) \norm{b_s}_{C^{0, \alpha}_{x}} \norm{v_s}_{C^{1, \alpha}_{x}}\, ds.
        \end{equation}
        Let us now fix $\varepsilon > 0$. Since $\sup_{0 \neq h \in \R^d} \eta_h(t) = \norm{v_t}_{C^{1, \alpha}_x}$, there exists for all $t \in (0, T)$ some $h^\varepsilon_t \in \R^d \setminus \{0\}$ such that
        \begin{equation*}
            \norm{v_t}_{C^{1, \alpha}_x} \leq \eta_{h^\varepsilon_t}(t) + \varepsilon.
        \end{equation*}
        Inserting this in \eqref{eq:pre_gronwall} yields
        \begin{equation} \label{eq:pre_gronwall_2}
            \norm{v_t}_{C^{1, \alpha}_x} \lesssim_{\alpha, d} \varepsilon + \int_0^t w_{h_t^\varepsilon}(t - s) \norm{f_s}_{C^{0, \alpha}_{x}}\, ds + \int_0^t w_{h_t^\varepsilon}(t - s) \norm{b_s}_{C^{0, \alpha}_{x}} \norm{v_s}_{C^{1, \alpha}_{x}}\, ds
        \end{equation}
        for all $t \in (0, T)$. Denote the first integral on the right-hand side by $I_t$, and let $p$ be such that $\nicefrac{1}{p} + \nicefrac{1}{q} = 1$. Towards applying Proposition~\ref{prop:gronwall}, we claim that $I \in L^\infty_t$ and $\sup_{t \in (0, T)}\norm{w_{h_t^\varepsilon}}_{L^p(0, t)} < \infty$. Beginning with $w_{h_t^\varepsilon}$, we have by Minkowski's inequality that for any $t \in (0, T)$ and $h \in \R^d \setminus \{0\}$, 
        \begin{equation*} \label{eq:w_bound}
            \begin{aligned}
                & \norm{w_h}_{L^p(0, t)} \leq \biggl(\int_0^t \bigl(1 + \frac{1}{r^\frac{(1-\alpha)}{2}}\bigr)^p\, dr\biggr)^\frac{1}{p} \\
                & \qquad + \biggl[|h|^{1-\alpha} \biggl(\int_{|h|^2}^{t\vee |h|^2} \frac{1}{r^{\frac{(2-\alpha)p}{2}}}\, dr\biggr)^{\frac{1}{p}} + |h|^{-\alpha}\biggl(\int_0^{|h|^2} \frac{1}{r^{\frac{(1-\alpha)p}{2}}}\, dr\biggr)^{\frac{1}{p}}\biggr] \indicator_{\{|h| \leq 1\}}.
            \end{aligned}
        \end{equation*}
        Up to some constant, the last two terms in the square bracket can be absolutely bounded by
        \begin{equation*}
            |h|^{1-\alpha} T^{\frac{1}{p} - 1 + \frac{\alpha}{2}} + |h|^{\frac{2}{p}-1},
        \end{equation*}
        and since $p \in (1, 2]$, the factor $|h|^{\nicefrac{2}{p} - 1}$ can be further bounded by $1$ for all $|h| \leq 1$. This means that $\norm{w_h}_{L^p(0, t)}$ is uniformly bounded for all $t \in (0, T)$ and $h \in \R^d\setminus \{0\}$, and consequently
        \begin{equation} \label{eq:w_bound_2}
            \sup_{t \in (0, T)}\norm{w_{h^\varepsilon_t}}_{L^p(0, t)} \leq C(\alpha, q, T) < \infty.
        \end{equation}
        For $I$, we then have by H\"older's inequality that
        \begin{equation} \label{eq:leading_term_bound}
            I_t = \int_0^t w_{h^\varepsilon_t}(t - s) \norm{f_s}_{C^{0, \alpha}_{x}}\, ds \leq \norm{w_{h^\varepsilon_t}}_{L^p(0, t)} \norm{f}_{L^q_t C^{0, \alpha}_{x}} \lesssim_{\alpha, q, T} \norm{f}_{L^q_t C^{0, \alpha}_{x}}
        \end{equation}
        for all $t \in (0, T)$. The modified Gronwall inequality from Proposition \ref{prop:gronwall} applied to \eqref{eq:pre_gronwall_2} now gives
        \begin{equation}\label{eq:c1_bound}
            \begin{aligned}
                \norm{v_t}_{C^{1, \alpha}_x} & \leq \gronwconst\biggl(C(\alpha, d)\sup_{s \in (0, t)}\norm{w_{h^\varepsilon_s}}_{L^p(0, s)}, \norm{b}_{L^q_t C^{0, \alpha}_x}\biggr) \bigl(\varepsilon + C(\alpha, d)\norm{I}_{L^\infty(0, t)}\bigr)\\
                & \leq \gronwconst\biggl(C(\alpha, d, q, T), \norm{b}_{L^q_t C^{0, \alpha}_x}\biggr) \bigl(\varepsilon + C(\alpha, d, q, T) \norm{f}_{L^q_t C^{0, \alpha}_{x}}\bigr),
            \end{aligned}
        \end{equation}
        where $\gronwconst$ is the constant defined in \eqref{eq:gronwconst} and $C(\alpha, d)$ is the implicit constant in~\eqref{eq:pre_gronwall_2}. Here, we have also used \eqref{eq:w_bound_2} and \eqref{eq:leading_term_bound}, and bounded the norms of $b$ and $f$ on the whole interval $(0, T)$. Since $\varepsilon$ was arbitrary, we pass $\varepsilon \to 0$ and denote the right-hand side by~$C_{f}^{(1)}$.
    \end{proof}

    It is clear from \eqref{eq:c1_bound} that $C^{(1)}_{f}$ is continuous in the arguments $\norm{f}_{L^q_t C^{0, \alpha}_{x}}$ and $\norm{b}_{L^q_t C^{0, \alpha}_{x}}$, and that $C^{(1)}_{f} \to 0$ when $\norm{f}_{L^q_t C^{0, \alpha}_{x}} \to 0$. Moreover, since $C^{(1)}_f$ is independent of $\lambda$, passing $\lambda \to \infty$ in \eqref{eq:v_bound} and \eqref{eq:dv_bound} yields $\lim_{\lambda \to \infty} \norm{v}_{L^\infty_t C^1_x} = 0$ (by dominated convergence).
     
    Towards proving higher order H\"older regularity, we next propose the following.

    \begin{claim} \label{claim:2}
        There is a constant $C_{f}^{(2)} = C_{f}^{(2)} \bigl(\alpha, d, q, T, \norm{b}_{L^q_t C^{0, \alpha}_{x}}, \norm{f}_{L^q_t C^{0, \alpha}_{x}}\bigr) \geq 0$ such that $\norm{v}_{L^q_t C^{2, \alpha}_{x}} \leq C_{f}^{(2)}$.
    \end{claim}

    \begin{proof}[Proof of Claim \ref{claim:2}]
        Let $h \in \R^d \setminus \{0\}$. Differentiating \eqref{eq:zvonkin_duhamel} and applying a second-order difference operator gives
        \begin{equation} \label{eq:second_order_diff}
            \begin{aligned}
                \Delta_h^2[\nabla v_t] & = \int_{0}^t \Delta_h^2\bigl[\nabla K_{t-s} * (f_s - b_s \cdot \nabla v_s)\bigr]\, ds \\
                & =  \int_{0}^{(t-|h|^2) \vee 0} \Delta_h^2\bigl[\nabla K_{t-s} * (f_s - b_s \cdot \nabla v_s)\bigr]\, ds \\
                & \quad + \int_{(t-|h|^2) \vee 0}^t \bigl(\Delta_h[\nabla K_{t-s}] * \Delta_h[f_s - b_s \cdot \nabla v_s]\bigr)\, ds \\
                & \eqqcolon I_t^{(1)} + I_t^{(2)}.
            \end{aligned}
        \end{equation}
        By Lemma \ref{lemma:heat_kernel_diff}, we have for any $g \in C^{0, \alpha}(\R^d)$ and $i = 1,...,d$ that
        \begin{equation*} \label{eq:kernel_mvt}
            \begin{aligned}
                \bignorm{\Delta_h^2[\partial_i K_{t-s} * g]}_{L^\infty_x} & = \biggnorm{\int_0^1 \biggl(\int_0^1 \nabla^2 (\partial_i K_{t-s} * g)(\cdot + (l+r)h) h\, dr \biggr)\cdot h \, dl}_{L^\infty_x} \\
                & \leq |h|^2 \norm{\nabla^2 (\partial_i K_t * g)}_{L^\infty_x} \lesssim_{\alpha, d} \frac{|h|^2}{(t-s)^{\frac{3-\alpha}{2}}} [g]_{C^{0, \alpha}_{x}}.
            \end{aligned}
        \end{equation*}
        Using this to estimate $I^{(1)} = I^{(1)}_t(x)$, we obtain
        \begin{equation*}
            \begin{aligned}
                \norm{I^{(1)}_t}_{L^\infty_x} & \leq \int_0^{(t-|h|^2)\vee 0} \norm{\Delta_h^2\bigl[\nabla K_{t-s} * (f_s - b_s \cdot \nabla v_s)\bigr]}_{L^\infty_x}\, ds \\
                & \lesssim_{\alpha, d} |h|^2 \int_0^{(t-|h|^2)\vee 0} \frac{1}{(t-s)^\frac{3 - \alpha}{2}} \bigl[f_s - b_s \cdot \nabla v_s\bigr]_{C^{0, \alpha}_x}\, ds \\
                & = |h|^2 \int_{|h|^2}^{t \vee |h|^2} \frac{1}{r^\frac{3 - \alpha}{2}} \bigl[f_{t-r} - b_{t-r} \cdot \nabla v_{t-r}\bigr]_{C^{0, \alpha}_x}\, dr,
            \end{aligned}
        \end{equation*}
        where we used a change of variables in the last equality. Minkowski's integral inequality gives
        \begin{equation} \label{eq:int_1}
            \begin{aligned}
                \norm{I^{(1)}}_{L^q_t L^\infty_x} & \lesssim_{\alpha, d} |h|^2 \biggl(\int_{0}^T \biggl(\int_{|h|^2}^{t\vee |h|^2} \frac{1}{r^\frac{3 - \alpha}{2}} \bigl[f_{t-r} - b_{t-r} \cdot \nabla v_{t-r}\bigr]_{C^{0, \alpha}_x}\, dr\biggr)^q dt\biggr)^{\frac{1}{q}}\\
                & = |h|^2 \biggl(\int_{|h|^2}^T \biggl(\int_{|h|^2}^t \frac{1}{r^\frac{3 - \alpha}{2}} \bigl[f_{t-r} - b_{t-r} \cdot \nabla v_{t-r}\bigr]_{C^{0, \alpha}_x}\, dr\biggr)^q dt\biggr)^{\frac{1}{q}}\\
                & \leq |h|^2 \int_{|h|^2}^T \frac{1}{r^\frac{3 - \alpha}{2}} \biggl(\int_{r}^T \bigl[f_{t-r} - b_{t-r} \cdot \nabla v_{t-r}\bigr]_{C^{0, \alpha}_x}^q\, dt \biggr)^\frac{1}{q}\, dr.
            \end{aligned}
        \end{equation}
        Since the inner integral can be bounded by
        \begin{equation*}
            \begin{aligned}
                \int_{r}^T \bigl[f_{t-r} - b_{t-r} \cdot \nabla v_{t-r}\bigr]_{C^{0, \alpha}_x}^q\, dt & \leq \int_{0}^T \bigl[f_t - b_t \cdot \nabla v_t\bigr]_{C^{0, \alpha}_x}^q\, dt \\
                & \lesssim_{q} \norm{f}_{L^q_t C^{0, \alpha}_{x}}^q + \norm{b}_{L^q_t C^{0, \alpha}_{x}}^q \norm{v}_{L^\infty_t C^{1, \alpha}_{x}}^q,
            \end{aligned}
        \end{equation*}
        we get that
        \begin{equation*}
            \norm{I^{(1)}}_{L^q_t L^\infty_x} \lesssim_{\alpha, d, q} |h|^2 \bigl(\norm{f}_{L^q_t C^{0, \alpha}_{x}} + \norm{b}_{L^q_t C^{0, \alpha}_{x}} \norm{v}_{L^\infty_t C^{1, \alpha}_{x}}\bigr) \int_{|h|^2}^T \frac{1}{r^\frac{3 - \alpha}{2}}\, dr.
        \end{equation*}
        For the second integral $I^{(2)} = I^{(2)}_t(x)$ in \eqref{eq:second_order_diff}, we have by Young's convolution inequality that
        \begin{equation*}
            \begin{aligned}
                \norm{I^{(2)}_t}_{L^\infty_x} & \leq \int_{(t-|h|^2)\vee 0}^t \norm{\Delta_h[\nabla K_{t-s}]}_{L^1_x} \norm{\Delta_h[f_s - b_s \cdot \nabla v_s]}_{L^\infty_x}\, ds \\
                & \lesssim |h|^\alpha \int_{(t-|h|^2)\vee 0}^t \frac{1}{\sqrt{t-s}} \bigl[f_s - b_s \cdot \nabla v_s\bigr]_{C^{0, \alpha}_x}\, ds \\
                & = |h|^\alpha \int_{0}^{|h|^2 \wedge t} \frac{1}{\sqrt{r}} \bigl[f_{t-r} - b_{t-r} \cdot \nabla v_{t-r}\bigr]_{C^{0, \alpha}_x}\, dr,
            \end{aligned}
        \end{equation*}
        where in the second line we used $\norm{\Delta_h[\nabla K_{t-s}]}_{L^1_x} \leq 2 \norm{\nabla K_{t-s}}_{L^1_x} \lesssim 1 / \sqrt{t-s}$. Applying Minkowski's inequality for integrals again yields
        \begin{equation} \label{eq:int_2}
            \begin{aligned}
                \norm{I^{(2)}}_{L^q_t L^\infty_x} & \lesssim |h|^\alpha \biggl(\int_0^T \biggl(\int_{0}^{|h|^2 \wedge t} \frac{1}{\sqrt{r}} [f_{t-r} - b_{t-r} \cdot \nabla v_{t-r}]_{C^{0, \alpha}_x}\, dr\biggr)^q\, dt \biggr)^{\frac{1}{q}} \\
                & \leq |h|^\alpha \int_{0}^{|h|^2} \frac{1}{\sqrt{r}} \biggl(\int_r^T [f_{t-r} - b_{t-r} \cdot \nabla v_{t-r}]_{C^{0, \alpha}_x}^q\, dt\biggr)^\frac{1}{q}\, dr \\
                & \lesssim_{q} |h|^{\alpha} \biggl(\norm{f}_{L^q_t C^{0, \alpha}_{x}} + \norm{b}_{L^q_t C^{0, \alpha}_{x}} \norm{v}_{L^\infty_t C^{1, \alpha}_{x}}\biggr) \int_{0}^{|h|^2} \frac{1}{\sqrt{r}}\, dr.
            \end{aligned}
        \end{equation}
        Inserting \eqref{eq:int_1} and \eqref{eq:int_2} in \eqref{eq:second_order_diff}, we obtain
        \begin{equation*}
            \begin{aligned}
                \bignorm{\Delta_h^2&[\nabla v]}_{L^q_t L^\infty_x}  \lesssim_{\alpha, d, q} \bigl(\norm{f}_{L^q_t C^{0, \alpha}_{x}} + \norm{b}_{L^q_t C^{0, \alpha}_{x}} \norm{v}_{L^\infty_t C^{1, \alpha}_{x}}\bigr)\\
                & \qquad\qquad\qquad\qquad \times \biggl(|h|^2 \int_{|h|^2}^T \frac{1}{r^\frac{3 - \alpha}{2}}\, dr + |h|^\alpha \int_{0}^{|h|^2} \frac{1}{\sqrt{r}}\, dr\biggr) \\
                & \leq |h|^{1+\alpha} \bigl(\norm{f}_{L^q_t C^{0, \alpha}_{x}} + \norm{b}_{L^q_t C^{0, \alpha}_{x}} \norm{v}_{L^\infty_t C^{1, \alpha}_{x}}\bigr) \biggl(\int_{1}^\infty \frac{1}{r^\frac{3 - \alpha}{2}}\, dr + \int_{0}^{1} \frac{1}{\sqrt{r}}\, dr\biggr).
            \end{aligned}
        \end{equation*}
        Dividing by $|h|^{1+\alpha}$ and inserting the result of Claim \ref{claim:1} gives
        \begin{equation} \label{eq:higher_Holder_bound}
            \frac{\bignorm{\Delta_h^2[\nabla v]}_{L^q_t L^\infty_x}}{|h|^{1+\alpha}} \leq C(\alpha, d, q) \bigl(\norm{f}_{L^q_t C^{0, \alpha}_{x}} + \norm{b}_{L^q_t C^{0, \alpha}_{x}} C^{(1)}_{f} \bigr).
        \end{equation}
        Denoting the right-hand side by $C_{f}^{(2)}$, the claim is proved by the equivalence \eqref{eq:zygmund_norm} of the H\"older and the Zygmund spaces for non-integer exponents.
    \end{proof}

    Since $C^{(2)}_{f}$ is given by the right-hand side of \eqref{eq:higher_Holder_bound}, we easily infer that $C^{(2)}_{f}$ is continuous with respect to $\norm{f}_{L^q_t C^{0, \alpha}_{x}}$ and $\norm{b}_{L^q_t C^{0, \alpha}_{x}}$, and that the limit \eqref{eq:apriori_limit_f} holds for $C^{(2)}_{f}$ as well.
    
    If $v$ is a solution of \eqref{eq:model_eq_weak}, then the weak derivative is given by the right-hand side of equation, and can thus be bounded by
    \begin{equation*}
        \begin{aligned}
            \norm{\partial_t v}_{L^q_t C^{0, \alpha}_x} & = \norm{\Delta v + f - b \cdot \nabla v - \lambda v}_{L^q_t C^{0, \alpha}_x} \\
            & \leq \norm{v}_{L^q_t C^{2, \alpha}_x} + \norm{f}_{L^q_t C^{0, \alpha}_x} + \norm{b}_{L^q_t C^{0, \alpha}_x} \norm{v}_{L^\infty_t C^{1, \alpha}_x} + \lambda \norm{v}_{L^q_t C^{0, \alpha}_x} \\
            & \leq (1+\lambda) C_{f}^{(2)} + \norm{f}_{L^q_t C^{0, \alpha}_x} + C_{\lambda, f}^{(1)} \norm{b}_{L^q_t C^{0, \alpha}_x} \eqqcolon C_f^{(3)}.
        \end{aligned}
    \end{equation*}
    Setting $C_{f} \coloneqq C_{f}^{(1)} + C_{f}^{(2)} + C_{f}^{(3)}$, we infer $\norm{v}_{\mathbb{W}^{1, q}_{2, \alpha}} \leq C_f$, as well as the limit~\eqref{eq:apriori_limit_f}. Uniqueness of solutions is now a consequence of the linearity of the equation.
\end{proof}

\subsection{Existence and stability}

We move on to the question about existence of solutions of \eqref{eq:model_eq_classical}, and stability with respect to perturbations of the coefficients $b$ and $f$.

\begin{theorem} \label{thm:pde_wellposedness}
    Assume that $b$ and $f$ satisfy \eqref{eq:h1}--\eqref{eq:h2}, and let $\lambda > 0$. Then there exists a unique solution $v \in \W^{1, q}_{2, \alpha}(T)$ of equation \eqref{eq:model_eq_classical}. For increasing values of $\lambda$, the solutions admit the limit
    \begin{equation} \label{eq:model_sol_limit_lambda}
        \lim_{\lambda \to \infty} \norm{v}_{L^\infty_t C^{1}_x} = 0.
    \end{equation}
    Furthermore, if $v^n$ and $v$ are the unique solutions in $\W^{1, q}_{2, \alpha}(T)$ corresponding to convergent sequences of coefficients
    \begin{equation} \label{eq:convergent_coefficients}
    \begin{aligned}
        & (b^n)_{n \in \N} \subset L^q_t C^{0, \alpha}_x; \qquad b^n \to b\ \text{in}\ L^q_t C^{0, \alpha}_x,\\
        & (f^n)_{n \in \N} \subset L^q_t C^{0, \alpha}_x; \qquad f^n \to f\ \text{in}\ L^q_t C^{0, \alpha}_x,
    \end{aligned}
    \end{equation}
    then
    \begin{equation} \label{eq:stability}
        \lim_{n \to \infty} v^n = v \quad \text{in}\quad W^{1, q}_{2, \alpha}(T).
    \end{equation}
\end{theorem}

\begin{proof}
    Suppose first that $v^n$ and $v$ are solutions corresponding to sequences of coefficients as in \eqref{eq:convergent_coefficients}. Then the difference $\tilde{v}^n \coloneqq v^n - v$ satisfies
    \begin{equation*}
        \partial_t \tilde{v}^n + b \cdot \nabla \tilde{v}^n + \lambda \tilde{v}^n = \Delta \tilde{v}^n + (f^n - f) + (b^n - b)\cdot \nabla v^n,
    \end{equation*}
    in the sense of Definition~\ref{def:model_eq_weak}. Proposition~\ref{prop:apriori_regularity} implies that $\lim_{n \to \infty} \norm{\tilde{v}^n}_{\mathbb{W}^{1, q}_{2, \alpha}} = 0$ in view of
    \begin{equation*}
        \bignorm{(f^n - f) + (b^n - b)\cdot \nabla v^n}_{L^q_t C^{0, \alpha}_x} \leq \norm{f^n - f}_{L^q_t C^{0, \alpha}_x} + \norm{b^n - b}_{L^q_t C^{0, \alpha}_x} \norm{v}_{L^\infty_t C^{1, \alpha}_x} \to 0
    \end{equation*}
    for $n \to \infty$. This proves \eqref{eq:stability}.
    
    To obtain existence for given coefficients $b$ and $f$, we introduce mollified coefficients $b^\varepsilon = (b * \rho^\varepsilon)$ and $f^\varepsilon = (f * \rho^\varepsilon)$, where $(\rho^\varepsilon)_{\varepsilon > 0} = (\rho^\varepsilon(x, t))_{\varepsilon > 0}$ is a family of standard mollifiers on $\R^{d+1}$, and we have extended $b$ and $f$ by zero outside $\R^d \times [0, T]$ when taking convolutions. Then for any $\varepsilon > 0$, there is a unique classical (smooth) solution $v^\varepsilon$ of \eqref{eq:model_eq_classical} with coefficients $b^\varepsilon$ and $f^\varepsilon$ (one can check that this solution also satisfies \eqref{eq:zvonkin_duhamel}). Furthermore, the mollified coefficients both converge in ${L^q_t C^{0, \alpha}_x}$ as ${\varepsilon \to 0}$, which means that $(v^\varepsilon)_{\varepsilon > 0}$ is a Cauchy sequence in $\mathbb{W}^{1, q}_{2, \alpha}(T)$. Denoting the limit by $v$, we easily verify that it is indeed a solution of~\eqref{eq:model_eq_classical}. The limit \eqref{eq:model_sol_limit_lambda} now follows directly from the same a priori limit.
\end{proof}

\section{Stochastic flows for drifts in \texorpdfstring{$L^q_t C^{0, \alpha}_x$}{TEXT}} \label{sec:sde}

In this section, we consider the forward SDE
\begin{equation} \label{eq:sde_integral_form}
    X_t^{x, s} = x + \int_s^t b_r(X_r^{x, s})\, dr + \int_s^t dW_r
\end{equation}
and the backward SDE
\begin{equation} \label{eq:sde_backward_integral_form}
    X_s^{x, t} = x - \int_s^t b_r(X_r^{x, t})\, dr - \int_s^t \hat{d}W_r,
\end{equation}
where we recall that $\hat{d}W_t$ denotes the backward Itô differential (here, we simply have $\int_s^t \hat{d}W_r = W_s - W_t$). As before, $(W_t)_{t \geq 0}$ is a $d$-dimensional Brownian motion on ${(\Omega, \F, \P)}$. The forward equation has to be solved on $(s, T)$, with a solution being a continuous stochastic process $(X_t^{x, s})_{t \in [s, T]}$ which is $\F_{s, t}$-adapted and satisfies \eqref{eq:sde_integral_form} $\P$-a.s. On the other hand, the backward equation is solved on~$(0, t)$, for a continuous process $(X_s^{x, t})_{s \in [0, t]}$ which is $\F_{s, t}$-adapted and satisfies \eqref{eq:sde_backward_integral_form}~\mbox{$\P$-a.s.} For both equations, uniqueness will always be understood in the pathwise sense.

In Sec.~\ref{sec:zvonkin}, we begin by showing that \eqref{eq:sde_integral_form}--\eqref{eq:sde_backward_integral_form} can be transformed to equations with more regular coefficients by a Zvonkin-type transformation. Then, in Sec.~\ref{sec:forward_backward_system}, we show that the transformed equations generate unique forward/backward flows of \diffeomorphisms. This is used in Sec.~\ref{sec:proof_main_flow} to prove Theorem~\ref{thm:main_flow}.

\subsection{A Zvonkin-type transformation} \label{sec:zvonkin}

We carry out some formal calculations which expand on those outlined in Sec.~\ref{sec:background_approach} in the introduction. For given ${x \in \R^d}$ and ${s \in [0, T]}$, assume that $(X_t^{x, s})_{t \in [s, T]}$ is a continuous $\R^d$-valued stochastic process which satisfies~\eqref{eq:sde_integral_form}. Let ${v\colon \R^d \times [0, T] \to \R^d}$ be the function whose components are solutions of the backward PDEs
\begin{equation} \label{eq:zvonkin_transform_componentwise}
    \begin{cases}
        \partial_t v^i + b \cdot \nabla v^i + \frac{1}{2} \Delta v^i = \lambda v^i - b^i, \\
        v^i_{t = T} \equiv 0
    \end{cases}
\end{equation}
for $i = 1, ..., d$. Then the process $(v_t(X_t^{x, s}))_{t \in [s, T]}$ is given by
\begin{equation} \label{eq:forward_ito_transformation_componentwise}
    \begin{aligned}
        v_t^i(X_t^{x, s}) & = v_s^i(x)+ \int_s^t \bigl(\partial_r v_r^i + b_r\cdot \nabla v_r^i + \frac{1}{2}\Delta v_r^i \bigr)(X_r^{x, s})\, dr \\
        &\quad  + \int_s^t \nabla v_r^i(X_r^{x, s})\cdot dW_r \\
        & = v_s^i(x) + \int_s^t \bigl(\lambda v_r^i - b_r^i\bigr)(X_r^{x, s})\, dr + \int_s^t \nabla v_r^i(X_r^{x, s})\cdot dW_r
    \end{aligned}
\end{equation}
for each component $v^i$ of $v$. Now set $g_t(x) \coloneqq x + v_t(x)$ and define 
\begin{equation*} \label{eq:forward_transformation}
    Y_t^{y, s} \coloneqq g_t(X_t^{x, s}), \qquad \text{where} \qquad  y = g_s(x).    
\end{equation*}
Using the fact that $X_t^{x, t}$ satisfies \eqref{eq:sde_integral_form} combined with \eqref{eq:forward_ito_transformation_componentwise}, we obtain
\begin{equation*} \label{eq:zvonkin_process}
    \begin{aligned}
        \bigl(Y_t^{y, s}\bigr)^i & = y^i + \lambda \int_s^t v_r^i(X_r^{x, s})\, dr + \int_s^t dW_r^i + \int_s^t \nabla v_r^i(X_r^{x, s})\cdot dW_r \\
        & = y^i + \lambda \int_s^t v_r^i(g_r^{-1}(Y_r^{y, s}))\, dr + \int_s^t dW_r^i + \int_s^t \nabla v_r^i(g_r^{-1}(Y_r^{y, s}))\cdot dW_r
    \end{aligned}
\end{equation*}
for each component $\bigl(Y_t^{x, s}\bigr)^i$ of $Y_t^{x, s}$. If we define a new set of coefficients
\begin{equation} \label{eq:forward_tranformed_coeffs}
    \tilde{b}_t \coloneqq \lambda v_t \circ g_t^{-1}, \qquad \tilde{\sigma}_t \coloneqq I+\nabla v_t \circ g_t^{-1},
\end{equation}
where $I$ denotes the identity matrix, the equation for $Y_t^{y, s}$ can be compactly written as
\begin{equation} \label{eq:zvonkin_process_forward_renamed}
    Y_t^{y, s} = y + \int_s^t \tilde{b}_r(Y_r^{y, s})\, dr + \int_s^t \tilde{\sigma}_r(Y_r^{y, s})\,dW_r.
\end{equation}

The above calculations can be rigorously justified at this point, provided that $b$ satisfies hypothesis \eqref{eq:h1}. Then by Theorem~\ref{thm:pde_wellposedness}, for any $\lambda > 0$ there exists a unique solution $v = (v^1, ..., v^d)$ of \eqref{eq:forward_ito_transformation_componentwise} which belongs to $\W^{2, \alpha}_{1, q}(T)$. In view of~\eqref{eq:model_sol_limit_lambda}, we can choose $\lambda$ large enough so that the function $g_t(x) = x + v_t(x)$ is invertible on~$\R^d$ for any $t \in [0, T]$. Indeed, since $\norm{v^i}_{L^\infty_t C^{1}_x}$ can be made arbitrarily small at the expense of $\lambda$  for $i = 1, ..., d$, the Jacobian matrix of $g$ will for large enough $\lambda$ become strictly diagonally dominant and therefore nonsingular, uniformly in space and time. By Hadamard's theorem~\cite[Theorem~V.59]{protter_2005}, this implies that $g$ is a \diffeomorphism on $\R^d$ for all $t \in [0, T]$. Furthermore, as a consequence of the regularity of $v$, we see that ${x \mapsto g_t(x)}$ is a continuously differentiable function with $\alpha$-H\"older continuous derivatives, uniformly in time. The same holds for $x \mapsto g_t^{-1}(x)$, by smoothness of the matrix inversion operator for nonsingular matrices. The application of Itô's formula in \eqref{eq:forward_ito_transformation_componentwise} is valid due to the regularity of~$v$ (this can be proved by approximation, see e.g.~\mbox{\cite[Theorem 3.7]{krylov_rockner_2005}}).

As a consequence of the regularity of $v$ and $g$, we can write down the following corollary.

\begin{corollary} \label{corr:forward_coeff_regularity}
    Assume that $b$ satisfies \eqref{eq:h1}. Then the coefficients in \eqref{eq:zvonkin_process_forward_renamed} are bounded, measurable, and fulfill
    \begin{equation*}
        \tilde{b} \in L^\infty((0, T); C^{1, \alpha}(\R^d;\R^d)), \qquad \tilde{\sigma} \in L^q((0, T); C^{1, \alpha}(\R^d; \R^{d\times d})). 
    \end{equation*}
\end{corollary}

Let us also consider a process $(X_s^{x, t})_{s \in [0, t]}$ satisfying the backward equation~\eqref{eq:sde_backward_integral_form}. If now $\hat{v}\colon \R^d \times [0, T] \to \R^d$ solves the forward system of PDEs
\begin{equation} \label{eq:zvonkin_transform_backward}
    \begin{cases}
        \partial_t \hat{v} + b \cdot \nabla \hat{v} + b = \frac{1}{2} \Delta \hat{v} + \lambda \hat{v}, \\
        v_{t = 0} \equiv 0,
    \end{cases}
\end{equation}
(understood \componentwise as in \eqref{eq:zvonkin_transform_componentwise}), then using Itô's formula in the backward variable for the composition $s \mapsto (\hat{v}_s \circ X_s^{x, t})$ yields
\begin{equation} \label{eq:backward_ito_transformation}
    \begin{aligned}
        \hat{v}_s(X_s^{x, t}) & = \hat{v}_t(x) - \int_s^t \bigl(\partial_r \hat{v}_r + b_r \cdot \nabla \hat{v}_r - \frac{1}{2} \Delta \hat{v}_r\bigr)(X_r^{x, t})\, dr \\
        & \quad - \int_s^t \nabla \hat{v}_r(X_r^{x, t})\cdot \hat{d}W_r \\
        & = \hat{v}_t(x) - \int_s^t \bigl(\lambda\hat{v}_r - b_r\bigr)(X_r^{x, t})\, dr - \int_s^t \nabla \hat{v}_r(X_r^{x, t})\cdot \hat{d}W_r,
    \end{aligned}
\end{equation}
where we have used that $\hat{v}$ is a solution of the system \eqref{eq:zvonkin_transform_backward}. Setting $\hat{g}_t(x) \coloneqq x + \hat{v}_t(x)$ allows us to rewrite \eqref{eq:backward_ito_transformation} as
\begin{equation*}
    \hat{g}_s(X_s^{x, t}) = \hat{g}_t(x) - \lambda \int_s^t \hat{v}_r(X_r^{x, t})\, dr - \int_s^t \bigl(I + \nabla \hat{v}_r(X_r^{x, t})\bigr)\cdot \hat{d}W_r.
\end{equation*}
Similarly to the forward equation, we define $\hat{Y}_s^{y, t} \coloneqq \hat{g}_s(X_s^{x, t})$ for $y = \hat{g}_t(x)$ and new coefficients
\begin{equation*} \label{eq:backward_tranformed_coeffs}
    \hat{b}_t \coloneqq \lambda \hat{v}_t \circ \hat{g}_t^{-1}, \qquad \hat{\sigma}_t \coloneqq I+\nabla \hat{v}_t \circ \hat{g}_t^{-1}.
\end{equation*}
Thus, we have obtained the system
\begin{equation} \label{eq:zvonkin_process_backward_renamed}
    \hat{Y}_s^{y, t} = y - \int_s^t \hat{b}_r(\hat{Y}_r^{y, t})\, dr - \int_s^t \hat{\sigma}_r(\hat{Y}_r^{y, t})\, \hat{d}W_r
\end{equation}
for the backward process $(\hat{Y}_s^{y, t})_{s \in [0, t]}$. Based on the previous discussion, we see that an analogue of Corollary~\ref{corr:forward_coeff_regularity} holds for the backward coefficients.

\begin{corollary} \label{corr:backward_coeff_regularity}
    Assume that $b$ satisfies \eqref{eq:h1}. Then the coefficients in \eqref{eq:zvonkin_process_backward_renamed} are bounded, measurable, and fulfill
    \begin{equation*}
        \hat{b} \in L^\infty((0, T); C^{1, \alpha}(\R^d; \R^d)), \qquad \hat{\sigma} \in L^q((0, T); C^{1, \alpha}(\R^d; \R^{d\times d})). 
    \end{equation*}
\end{corollary}

\subsection{Forward/backward equations with coefficients in \texorpdfstring{$L^q_t C^{1, \alpha}_x$}{TEXT}} \label{sec:forward_backward_system}

Based on the preceding calculations in Sec.~\ref{sec:zvonkin}, we will assume that we are given deterministic bounded and measurable functions $a$ and $\sigma$ which satisfy
\begin{equation} \label{eq:h3}
    \tag{H3}
    a \in L^1((0, T); C^{1, \alpha}(\R^d; \R^d)), \quad \sigma \in L^2((0, T); C^{1, \alpha}(\R^d; \R^{d\times d})), \quad \alpha \in (0, 1).
\end{equation}
Due to Corollaries \ref{corr:forward_coeff_regularity} and \ref{corr:backward_coeff_regularity}, the coefficients $\tilde{b}, \tilde{\sigma}$ of the forward equation \eqref{eq:zvonkin_process_forward_renamed}, and the coefficients $\hat{b}, \hat{\sigma}$ of the backward equation \eqref{eq:zvonkin_process_backward_renamed}, both fulfill \eqref{eq:h3}. Since we shall not be interested in optimizing the temporal regularity of solutions, we only assume ${a \in L^1_t C^{1, \alpha}_x}$ and ${\sigma \in L^2_t C^{1, \alpha}_x}$ rather than higher integrability in time: these assumptions represent the minimal integrability conditions necessary for the forthcoming results to hold.

Consider the forward model equation
\begin{equation} \label{eq:general_system_forward}
    Z_t^{x, s} = x + \int_s^t a_r(Z_r^{x, s})\, dr + \int_s^t \sigma_r(Z_r^{x, s})\, dW_r
\end{equation}
for $0 \leq s \leq t \leq T$ and $x \in \R^d$, where the solution is denoted by $Z_t$ so that it is easily distinguished from the solution of the original SDE \eqref{eq:sde_integral_form}.  The following result is due to H.~Kunita (see Theorems 3.4.1, 4.5.1 and 4.6.5 in \cite{kunita_1990}).

\begin{theorem} \label{thm:kunita_flow_theorem}
    For any $x \in \R^d$ and $s \in [0, T]$, there exists a unique solution $(Z_t^{x, s})_{t \in [s, T]}$ of equation \eqref{eq:general_system_forward}. Moreover, for any~${\beta < \alpha}$, the system of solutions $(Z_t^{x, s}\setsep 0 \leq s \leq t \leq T,\, x \in \R^d)$ has a modification, denoted by $Z = Z_{s, t}(x)$, which is a forward stochastic flow of $C^{1, \beta}$-\diffeomorphisms. For any $x \in \R^d$ and $s \in [0, T]$, the process $(Z_{s, t}(x))_{t \in [s, T]}$ is a forward $C^{1, \beta}$-\semimartingale.
\end{theorem}

\begin{remark}
    Kunita's notation in \cite{kunita_1990} differs from the one we use here; in particular, note that
    \begin{equation*}
        a_t(x)\, dt + \sigma_t(x)\, dW_t = F(x, dt)
    \end{equation*}
    and $B^{k, \beta}_b$ is exactly the space of functions with $C^{k, \beta}$-norm integrable over $[0, T]$.
\end{remark}

Next, we prove a result concerning the approximation of coefficients and convergence of the flow given hypothesis \eqref{eq:h3}.

\begin{lemma} \label{lemma:forward_stability}
    Let $(Z^n)_{n \in \N}$ be forward stochastic flows of $C^{1, \beta}$-\diffeomorphisms generated by \eqref{eq:general_system_forward} with coefficients
    \begin{equation}\label{eq:general_convergent_coeffs}
        \begin{aligned}
            & (a^n)_{n \in \N} \subset L^1_t C^{1, \alpha}_x;\qquad \lim_{n \to \infty} a^n = a\quad \text{in}\quad L^1_t C^{1}_x, \\
            & (\sigma^n)_{n \in \N} \subset L^2_t C^{1, \alpha}_x;\qquad \lim_{n \to \infty} \sigma^n = \sigma \quad \text{in}\quad L^2_t C^{1}_x.
        \end{aligned}
    \end{equation}
    Then for any $p > 0$, the sequence of flows $(Z^n)_n$ converges to $Z$ in the sense that
    \begin{align}
        \lim_{n \to \infty} \sup_{x \in \R^d} \sup_{s \in (0, T)} \E \biggl[ \sup_{t \in (s, T)} |Z_{s, t}^n(x) - Z_{s, t}(x)|^p\biggr] = 0, \label{eq:flow_stability} \\
        \lim_{n \to \infty} \sup_{x \in \R^d} \sup_{s \in (0, T)} \E \biggl[\sup_{t \in (s, T)} \norm{\nabla Z_{s, t}^n(x) - \nabla Z_{s, t}(x)}^p\biggr] = 0. \label{eq:gradient_flow_stability}
    \end{align}
\end{lemma}

\begin{proof}
    Before delving into estimates, we make some simplifications. First, note that we can take $s = 0$ without loss of generality: using that $Z_{s, t}^n$ and $Z_{s, t}$ are homeomorphims on $\R^d$ in addition to the flow property implies
    \begin{equation*}
        \sup_{x \in \R^d} \sup_{s \in (0, T)} \E \biggl[ \sup_{t \in (s, T)} |Z_{s, t}^n(x) - Z_{s, t}(x)|^p\biggr] = \sup_{x \in \R^d} \E \biggl[ \sup_{t \in (0, T)} |Z_{t}^n(x) - Z_{t}(x)|^p\biggr]
    \end{equation*}
    for all $n \in \N$. Furthermore, it is enough to prove the assertion for $p \geq 2$; if this is not the case, one can simply use H\"older's inequality to increase the exponent inside the expectation. Finally, we will assume that $a^n$ and $\sigma^n$ are coefficients in~$L^\infty_t C^{1, \alpha}_x$ which converges to $a$ and $\sigma$ in $L^\infty_t C^{1}_x$. In all of the subsequent integrals, the general case can be reduced to this case by a time-change. Let us demonstrate this idea for $(Z_t^x)_{t \in [0, T]}$, being a continuous stochastic process of the flow $Z_t(x)$ which satisfies \eqref{eq:general_system_forward} with $s=0$. Define 
    \begin{equation*}
        A_t \coloneqq \int_0^t \bigl(\norm{a_r}_{C^{1, \alpha}_x} + \norm{\sigma_r}_{C^{1, \alpha}_x}^2 + 1\bigr)\, dr, \qquad \tau_t \coloneqq \sup\{s \in [0, T]\setsep A_s \leq t\},
    \end{equation*}
    so that $\tau_t$ is the inverse of the strictly increasing function $A_t$. Then $\bar{Z}_t \coloneqq Z_{\tau_t}$ is also a $C^{1, \beta}$-\semimartingale with respect to the filtration $\bar{\F}_t \coloneqq \F_{\tau_t}$, and it satisfies
    \begin{equation*}
        \bar{Z}_t^x = x + \int_0^t \bar{a}_r(\bar{Z}_r^x)\, dr + \int_0^t \bar{\sigma}_r(\bar{Z}_r^x)\, dW_r
    \end{equation*}
    $\P$-a.s, where the coefficients are given by
    \begin{equation*}
        \bar{a}_t(x) = \frac{a_{\tau_t}(x)}{\norm{a_{\tau_t}}_{C^{1, \alpha}_x} + \norm{\sigma_{\tau_t}}_{C^{1, \alpha}_x}^2 + 1}, \qquad \bar{\sigma}_t(x) = \frac{\sigma_{\tau_t}(x)}{\norm{a_{\tau_t}}_{C^{1, \alpha}_x} + \norm{\sigma_{\tau_t}}_{C^{1, \alpha}_x}^2 + 1}.
    \end{equation*}
    These time-changed coefficients are bounded in $C^{1, \alpha}_x$ uniformly in time. (Note also that if \eqref{eq:general_convergent_coeffs} holds, then $\bar{a}^n \to \bar{a}$ and $\bar{\sigma}^n \to \bar{\sigma}$ in $L^\infty_t C^1_x$ as $n \to \infty$). For more details about the this method, see \mbox{\cite[Theorem~3.2.9]{kunita_1990}}. 
    
    For a clear organization, we split the proof into three claims.
    
    \begin{claim} \label{claim:stability_claim_1}
    There is a constant $C^{(1)} = C^{(1)}(d, p, T, \norm{a}_{L^\infty_t C^1_x}, \norm{\sigma}_{L^\infty_t C^1_x}) \geq 0$ such that
    \begin{equation*}
        \sup_{x \in \R^d} \E\biggl[\sup_{t \in (0, T)} |Z_{t}^n(x) - Z_{t}(x)|^p\biggr] \leq C^{(1)} \bigl(\norm{a^n - a}_{L^\infty_t C^0_x}^p + \norm{\sigma^n - \sigma}_{L^\infty_t C^0_x}^p\bigr).
    \end{equation*}
    for all $n \in \N$.    
    \end{claim}
    
    \begin{proof}
    The difference $Z_t^n(x) - Z_t(x)$ satisfies $\P$-a.s.
    \begin{equation*}
        Z_{t}^n(x) - Z_t(x) = \int_{0}^{t} a_r^n(Z_r^n(x)) - a_r(Z_r(x))\, dr + \int_{0}^{t} \sigma_r^n(Z_r^n(x)) - \sigma_r(Z_r(x))\, dW_r
    \end{equation*}
    Using Doob's maximal inequality and the Burkholder--Davies--Gundy (BDG) inequality, one can estimate, for any $t \in (0, T)$,
    \begin{equation*}
        \begin{aligned}
            & \E\biggl[\sup_{u \in (0, t)} |Z_{u}^n - Z_{u}|^p\biggr]\\
            & \lesssim_{p} \int_{0}^{t} \E\bigl[|a_r^n(Z^n_r) - a_r(Z_r)|^p\bigr]\, dr + \sum_{j = 1}^d \int_{0}^{t} \E\bigl[|\sigma_r^{n, j}(Z_r^n) - \sigma_r^{j}(Z_r)|^p\bigr]\, dr \\
            & \lesssim_{p, T, d} \norm{a^n - a}_{L^\infty_t C^0_x}^p + \norm{\sigma^n - \sigma}_{L^\infty_t C^0_x}^p \\
            & \qquad + \bigl(\norm{a}_{L^\infty_t C^1_x}^p + \norm{\sigma}_{L^\infty_t C^1_x}^p\bigr) \int_0^t \E\biggl[\sup_{u \in (0, r)}|Z_u^n(x) - Z_u(x)|^p\biggr]\, dr,
        \end{aligned}
    \end{equation*}
    where $\sigma^{n,j }$ denotes the $j$-th column vector of $\sigma^n$. Now \eqref{claim:stability_claim_1} follows upon application of Gronwall's inequality. 
    \end{proof}
    
    Since we have assumed that $a^n \to a$ and $\sigma^n \to \sigma$ in $L^\infty_t C^{1}_x$, this already proves \eqref{eq:flow_stability}. Next, we show that the derivatives of the flow are uniformly bounded in expectation.
    
    \begin{claim} \label{claim:stability_claim_2}
    There is a constant
    \begin{equation*}
        C^{(2)} = C^{(2)}\bigl(d, p, T, \sup_n \norm{a^n}_{L^\infty_t C^1_x}, \sup_n \norm{\sigma^n}_{L^\infty_t C^1_x}\bigr) \geq 0
    \end{equation*}
    such that
    \begin{equation} \label{eq:gradient_flow_boundedness}
        \sup_{x \in \R^d} \E\biggl[\sup_{t \in (0, T)}|\partial_i Z_t^n(x)|^p\biggr] \leq C^{(2)}
    \end{equation}
    for all $n \in \N$ and $i = 1, ..., d$. 
    \end{claim}

    \begin{proof}
    Since the forward stochastic flow $Z_{t}^n(x)$ is a $C^{1, \beta}$-\diffeomorphism, the $i$-th derivative $(\partial_i Z_t^n(x))_{t \in [0, T]}$ satisfies
    \begin{equation*}
        \partial_i Z_t^n = e_i + \int_{0}^t \nabla a_r^n(Z_r^n) \partial_i Z_r^n\, dr + \sum_{j = 1}^d \int_{0}^t \nabla \sigma_r^{n, j}(Z_r^n) \partial_i Z_r^n\, dW^j_r
    \end{equation*}
    $\P$-a.s (see \cite[Chap.~4.6]{kunita_1990}). We infer by Doob's inequality and the BDG inequality that
    \begin{equation*}
        \begin{aligned}
            & \E\biggl[\sup_{u \in (0, t)} |\partial_i Z_u^n(x)|^p \biggr] \\
            & \quad \lesssim_{d, p, T} 1 + \bigl(\norm{a^n}_{L^\infty_t C^{1}_x}^p + \norm{\sigma^n}_{L^\infty_t C^{1}_x}^p\bigr) \int_0^t \E\biggl[\sup_{u\in (0, r)}|\partial_i Z_u^n(x)|^p\biggr]\, dr.
        \end{aligned}
    \end{equation*}
    Gronwall's inequality now implies \eqref{eq:gradient_flow_boundedness}.
    \end{proof}
    
    Finally, me make the following claim, which would imply \eqref{eq:gradient_flow_stability} by convergence of the coefficients in $L^\infty_t C^1_x$.

    \begin{claim}
    There is a constant 
    \begin{equation*}
        C^{(3)} = C^{(3)}\bigl(\alpha, d, p, T, \sup_n \norm{a^n}_{L^\infty_t C^1_x}, \sup_n \norm{\sigma^n}_{L^\infty_t C^1_x}, \norm{a}_{L^\infty_t C^{1, \alpha}_x}, \norm{\sigma}_{L^\infty_t C^{1, \alpha}_x}\bigr) \geq 0    
    \end{equation*}
    such that
    \begin{equation} \label{eq:gradient_flow_stability_claim}
        \sup_{x \in \R^d} \E\biggl[\sup_{t \in (0, T)}|\partial_i Z_t^n(x) - \partial_i Z_t(x)|^p\biggr] \leq C^{(3)} \bigl(\norm{a^n - a}_{L^\infty_t C^1_x}^p + \norm{\sigma^n - \sigma}_{L^\infty_t C^1_x}^p\bigr)
    \end{equation}
    for all $n \in \N$ and $i = 1, ..., d$.    
    \end{claim}

    \begin{proof}
    Here, we will use that
    \begin{equation*}
        \begin{aligned}
            \partial_i Z_t^n(x) - \partial_i Z_t(x) & = \int_{0}^t \nabla a_r^n(Z_r^n(x)) \partial_i Z_r^n(x) - \nabla a(Z_r(x)) \partial_i Z_r(x)\, dr \\
            & \quad + \sum_{j = 1}^d \int_{0}^t \nabla \sigma^{n, j}(Z_r^n(x)) \partial_i Z_r^n(x) - \nabla \sigma^j(Z_r(x)) \partial_i Z_r(x)\, dW^j_r
        \end{aligned}
    \end{equation*}
    $\P$-a.s., and moreover
    \begin{equation*}
        \begin{aligned}
            & |\nabla a_t^n(Z_t^n) \partial_i Z_t^n - \nabla a(Z_t) \partial_i Z_t| \\
            & \quad \leq \norm{a_t^n - a_t}_{C^1_x} |\partial_i Z_t^n| + \norm{a_t}_{C^{1, \alpha}_x}|Z_t^n - Z_t|^{\alpha} |\partial_i Z_t^n| + \norm{a_t}_{C^1_x} |\partial_i Z_t^n - \partial_i Z_t|
        \end{aligned}
    \end{equation*}
    and similarly
    \begin{equation*}
        \begin{aligned}
            & |\nabla \sigma_t^{j, n}(Z_t^n) \partial_i Z_t^n - \nabla \sigma_t^{j, n}(Z_t) \partial_i Z_t| \\
            & \quad \leq \norm{\sigma_t^{n, j} - \sigma_t^j}_{C^1_x} |\partial_i Z_t^n| + \norm{\sigma_t^j}_{C^{1, \alpha}_x}|Z_t^n - Z_t|^{\alpha} |\partial_i Z_t^n| + \norm{\sigma_t^j}_{C^1_x} |\partial_i Z_t^n - \partial_i Z_t|
        \end{aligned}
    \end{equation*}
    for $j = 1, ..., d$. Using Doob's inequality and the BDG inequality, we obtain
    \begin{equation*}
        \begin{aligned}
            & \E\biggl[\sup_{u \in (0, t)} |\partial_i Z_u^n - \partial_i Z_u|^p\biggr] \lesssim_{p} \int_{0}^t \E\bigl[|\nabla a_r^n(Z_r^n) \partial_i Z_r^n - \nabla a(Z_r) \partial_i Z_r|^p\bigr]\, dr \\
            & \qquad + \sum_{j = 1}^d \int_{0}^t \E\bigl[|\nabla \sigma^{n, j}(Z_r^n) \partial_i Z_r^n - \nabla \sigma^j(Z_r) \partial_i Z_r|^p\bigr]\, dr \\
            & \quad \lesssim_{p} \biggl(\norm{a^n - a}_{L^\infty_t C^1_x}^p + \sum_{j = 1}^d \norm{\sigma^{n, j} - \sigma^j}_{L^\infty_t C^1_x}^p\biggr) \int_0^t \E\bigl[|\partial_i Z_r^n|^p\bigr]\, dr \\
            & \qquad + \biggl(\norm{a}_{L^\infty_t C^{1, \alpha}_x}^p + \sum_{j = 1}^d \norm{\sigma^j}_{L^\infty_t C^{1, \alpha}_x}^p\biggr) \int_0^t \E\bigl[|Z_t^n - Z_r|^{\alpha p} |\partial_i Z_r^n|^p \bigr]\, dr \\
            & \qquad + \biggl(\norm{a}_{L^\infty_t C^{1}_x}^p + \sum_{j = 1}^d \norm{\sigma^j}_{L^\infty_t C^{1}_x}^p\biggr) \int_0^t \E\biggl[\sup_{u \in (0, r)}|\partial_i Z_u^n - \partial_i Z_u|^p\biggr]\, dr.
        \end{aligned}
    \end{equation*}
    Using H\"older's inequality in the penultimate integral and then applying Claim \ref{claim:stability_claim_1} and \ref{claim:stability_claim_2} yields
    \begin{equation*}
        \begin{aligned}
            & \E\biggl[\sup_{u \in (0, t)} |\partial_i Z_u^n - \partial_i Z_u|^p\biggr] \lesssim_{d, p, T} C^{(2)}\bigl(\norm{a^n - a}_{L^\infty_t C^1_x}^p + \norm{\sigma^{n} - \sigma}_{L^\infty_t C^1_x}^p\bigr) \\
            & \qquad + C^{(1)}C^{(2)} \bigl(\norm{a}_{L^\infty_t C^{1, \alpha}_x}^p + \norm{\sigma}_{L^\infty_t C^{1, \alpha}_x}^p\bigr) \bigl(\norm{a^n - a}_{L^\infty_t C^0_x}^p + \norm{\sigma^n - \sigma}_{L^\infty_t C^0_x}^p\bigr) \\
            & \qquad + \bigl(\norm{a}_{L^\infty_t C^{1}_x}^p + \norm{\sigma}_{L^\infty_t C^{1}_x}^p\bigr) \int_0^t \E\biggl[\sup_{u \in (0, r)}|\partial_i Z_u^n - \partial_i Z_u|^p\biggr]\, dr.
        \end{aligned}
    \end{equation*}
    Application of Gronwall's inequality proves \eqref{eq:gradient_flow_stability_claim}.
    \end{proof}
    This completes the proof of Lemma \ref{lemma:forward_stability}.
\end{proof}

Next, we consider the backward model SDE
\begin{equation} \label{eq:general_system_backward}
    Z_s^{x, t} = x - \int_s^t a_r(Z_r^{x, t})\, dr - \int_s^t \sigma_r(Z_r^{x, t})\, \hat{d}W_r,
\end{equation}
for $x \in \R^d$ and $0 \leq s \leq t \leq T$. A solution should be viewed as a backward continuous stochastic process $(Z_{s}^{x, t})_{s \in [0, t]}$ adapted to $\F_{s, t}$. Since the assumption~\eqref{eq:h3} on the coefficients is the same here as for the forward equation, we have completely parallell results in the backward direction (again, see \cite{kunita_1990}).

In the following theorem, we denote the backward flow generated by \eqref{eq:general_system_backward} by $\hat{Z} = \hat{Z}_{s, t}(x)$. The reason for this is that the backward flow does not in general equal the inverse of the forward flow given in Theorem~\ref{thm:kunita_flow_theorem}. We will point this out whenever there is possible confusion in the succeeding sections.

\begin{theorem} \label{thm:kunita_flow_theorem_backward}
    Let coefficients $a$ and $\sigma$ be bounded, measurable and satisfy \eqref{eq:h3}.
    \begin{itemize}
        \item[\textit{(i)}] For any $x \in \R^d$ and $t \in [0, T]$, there exists a unique solution $(Z_s^{x, t})_{s \in [0, t]}$ of the backward equation \eqref{eq:general_system_backward}. Moreover, for any $\beta < \alpha$, the system of solutions $(Z_s^{x, t}\setsep 0 \leq s \leq t \leq T,\, x \in \R^d)$ has a modification, denoted by $\hat{Z}_{s, t}(x)$, which is a backward stochastic flow of $C^{1, \beta}$-\diffeomorphisms. For any $x \in \R^d$ and $s \in [0, T]$, the process $(\hat{Z}_{s, t}(x))_{s \in [0, t]}$ is a backward $C^{1, \beta}$-\semimartingale.
    
        \item[\textit{(ii)}] Let $(\hat{Z^{n}})_{n \in \N}$ and $\hat{Z}$ be backward stochastic flows of $C^{1, \beta}$-\diffeomorphisms corresponding to convergent sequences of coefficients as in \eqref{eq:general_convergent_coeffs}. Then for any $p > 0$, the sequence of flows $(\hat{Z}^{n})_n$ converges to $\hat{Z}$ in the sense that
        \begin{align*}
            \lim_{n \to \infty} \sup_{x \in \R^d} \sup_{t \in (0, T)} \E \biggl[ \sup_{s \in (0, t)} |\hat{Z}_{s, t}^{n}(x) - \hat{Z}_{s, t}(x)|^p\biggr] = 0, \\ %\label{eq:backward_flow_stability} \\
            \lim_{n \to \infty} \sup_{x \in \R^d} \sup_{t \in (0, T)} \E \biggl[\sup_{s \in (0, t)} \norm{\nabla \hat{Z}_{s, t}^{n}(x) - \nabla \hat{Z}_{s, t}(x)}^p\biggr] = 0. %\label{eq:backward_gradient_flow_stability}
        \end{align*}
    \end{itemize}
\end{theorem}

\subsection{Proof of Theorem \ref{thm:main_flow}} \label{sec:proof_main_flow}

We will now consolidate the results of the two previous sections into a proof of Theorem~\ref{thm:main_flow}. Let us first note that after applying the Zvonkin-type transformation to \eqref{eq:sde_integral_form}, the transformed equation \eqref{eq:zvonkin_process_forward_renamed} possesses all the desired properties:

\begin{lemma} \label{lemma:zvonkin_flow}
    For all $y \in \R^d$ and $s \in [0, T]$, there exists a unique solution $(Y_t^{y, s})_{t \in [s, T]}$ of \eqref{eq:zvonkin_process_forward_renamed} on $(s, T)$. For any $\beta < \alpha$, the system of such solutions has a modification $Y = Y_{s, t}(y)$ which is a forward stochastic flow of $C^{1, \beta}$-\diffeomorphisms on $\R^d$. Moreover, if $(b^n)_{n \in \N}$ is a sequence of coefficients for \eqref{eq:sde_integral_form} satisfying~\eqref{eq:convergent_coeff} for some $\alpha' > 0$, then the corresponding forward flows $Y^n$ converge to $Y$ in the sense of \mbox{\eqref{eq:forward_stability}--\eqref{eq:forward_gradient_stability}}.
\end{lemma}

\begin{proof}
    The first part of the claim is a direct consequence of Theorem~\ref{thm:kunita_flow_theorem}. Let $Y^n$ be the forward stochastic flows of $C^{1, \beta}$-\diffeomorphisms generated by~\eqref{eq:zvonkin_process_forward_renamed} with coefficients
    \begin{equation*}
        \tilde{b}_t^n \coloneqq \lambda v_t^n \circ (g_t^n)^{-1}, \qquad \tilde{\sigma}_t^n \coloneqq I+\nabla v_t^n \circ (g_t^n)^{-1},
    \end{equation*}
    where $v^n$ are the unique solutions of \eqref{eq:zvonkin_transform_componentwise} in $\W^{1, q}_{2, \alpha}(T)$ corresponding to $b^n$, and ${g_t^n(x) \coloneqq x + v_t^n(x)}$. The parameter $\lambda$ is fixed independently of $n$, so that all $g^n$ are invertible on $\R^d$. To prove that $Y^n$ converges to $Y$ in the sense of \mbox{\eqref{eq:forward_stability}--\eqref{eq:forward_gradient_stability}}, we will show that
    \begin{equation} \label{eq:convergence_transformed_coeffs}
        \lim_{n \to \infty} \tilde{b}^n = \tilde{b} \quad \text{in}\quad  L^\infty_t C^{1, \alpha'}_x, \qquad \lim_{n \to \infty} \tilde{\sigma}^n = \tilde{\sigma}\quad \text{in}\quad L^q_t C^{1, \alpha'}_x
    \end{equation}
    where $\tilde{b}$ and $\tilde{\sigma}$ are given by \eqref{eq:forward_tranformed_coeffs}, and then infer convergence in view of Lemma \ref{lemma:forward_stability}.
    
    Assume without loss of generality that $\alpha' \leq \alpha$ (since $C^{\alpha'}_x \subset C^\alpha_x$ if $\alpha' > \alpha$). By Theorem~\ref{thm:pde_wellposedness}, the convergence of $b^n$ to $b$ in $L^q_t C^{0, \alpha'}_x$ implies that $v^n$ converges to $v$ in~$W^{1, q}_{2, \alpha'}(T)$. Furthermore, note that for a.e.~$t \in [0, T]$, for any $n \in \N$ and $x \in \R^d$, there exists $\xi \in \R^d$ by the mean value theorem such that
    \begin{equation*}
        \bigl(g_t^n\circ (g_t^n)^{-1}\bigr)(x) - \bigl(g_t^n\circ g_t^{-1}\bigr)(x) = \nabla g_t^n(\xi) \bigl((g_t^n)^{-1}(x) - g_t^{-1}(x)\bigr).
    \end{equation*}
    Since $\nabla g_t^n$ is invertible on $\R^d$, we obtain the estimate
    \begin{equation*}
        \begin{aligned}
            \bigl|(g_t^n)^{-1}(x) - g_t^{-1}(x)\bigr| & = \bigl|(\nabla g_t^n(\xi))^{-1} \bigl(x - \bigl(g_t^n\circ g_t^{-1}\bigr)(x)\bigr)\bigr| \\
            & \leq \norm{(\nabla g_t^n)^{-1}}_{C^0_x} \bigl|\bigl(g_t\circ g_t^{-1}\bigr)(x) - \bigl(g_t^n \circ g_t^{-1}\bigr)(x)\bigr| \\
            & = \norm{(\nabla g_t^n)^{-1}}_{C^0_x} \bigl|\bigl(v_t \circ g_t^{-1}\bigr)(x) - \bigl(v_t^n \circ g_t^{-1}\bigr)(x)\bigr|,
        \end{aligned}
    \end{equation*}
    from which convergence of $(g^n)^{-1}$ to $g^{-1}$ in $L^\infty_t C^{0, \alpha'}_x$ follows (note that $g_t^{-1}$ and $(g_t^n)^{-1}$ do not belong to $L^\infty_t C^{0, \alpha'}_x$ since they are unbounded, but the unbounded parts cancel when taking the difference). Convergence of $(g^n)^{-1}$ to $g^{-1}$ in $L^\infty_t C^{1, \alpha'}_x$ can now be obtained via the formula
    \begin{equation*}
        \bigl(\nabla (g_t^n)^{-1}\bigr)(x) = \bigl[(\nabla g_t^n \circ (g_t^n)^{-1})(x)\bigr]^{-1}
    \end{equation*}
    which is due to the inverse function theorem. Combined with the convergence $v^n \to v$ in $\W^{1, q}_{2, \alpha'}(T)$, this proves \eqref{eq:convergence_transformed_coeffs} and thereby the claim.
\end{proof}

\begin{proof}[Proof of Theorem~\ref{thm:main_flow}]
    We begin with the proof of part (i), for     which we will have to show that the properties from Lemma \ref{lemma:zvonkin_flow} also hold for the original equation \eqref{eq:sde_integral_form}. First, if for any $x \in \R^d$ and $s \in [0, T]$ there are two distinct solutions $(X_t^{x, s})_{t \in [s, T]}$ and $(\bar{X}_t^{x, s})_{t \in [s, T]}$ of \eqref{eq:sde_integral_form}, then by the calculations in Sec.~\ref{sec:zvonkin}, both $g_t(X_t^{x, s})$ and $g_t(\bar{X}_t^{x, s})$ satisfy \eqref{eq:zvonkin_process_forward_renamed} with initial condition $y = g_s^{-1}(x)$. This is a contradiction, since pathwise uniqueness holds for \eqref{eq:zvonkin_process_forward_renamed}, and we conclude that pathwise uniqueness also holds for \eqref{eq:sde_integral_form}. Combined with the existence of weak solutions which can be obtained by Girsanov's theorem, this gives existence of strong solutions for all $x \in \R^d$ and $s \in [0, T]$ by the Yamada--Watanabe theorem (see e.g.~\cite{karatzas_1996}).
    
    Let $X = X_{s, t}(x)$ be defined as
    \begin{equation*}
        X_{s, t}(x) \coloneqq \bigl(g_t^{-1} \circ Y_{s, t} \circ g_s\bigr)(x).
    \end{equation*}
    Then for all $x \in \R^d$ and $s \in [0, T]$, the continuous stochastic process $(X_{s, t}(x))_{t \in [s, T]}$ coincides with the solution $(X_t^{x, s})_{t \in [s, T]}$ of \eqref{eq:sde_integral_form} $\P$-a.s. Moreover, $X$ is a forward stochastic flow of $C^{1, \beta}$-\diffeomorphisms. Indeed, we have \mbox{$\P$-a.s.}~that
    \begin{equation*}
        \bigl(X_{r, t} \circ X_{s, r}\bigr)(x) = \bigl(g_t^{-1} \circ Y_{r, t} \circ g_r \circ g_r^{-1} \circ Y_{s, r} \circ g_s\bigr)(x) = X_{s, t}(x)
    \end{equation*}
    for all $x \in \R^d$ and $0 \leq s \leq r \leq t \leq T$, and furthermore 
    \begin{equation*}
        X_{s, s}(x) = g_s^{-1}(Y_{s, s}(g_s(x))) = x.
    \end{equation*}
    Also, $x \mapsto X_{s, t}(x)$ inherits the \diffeomorphism property and the regularity of $Y_{s, t}$, $g_s$ and $g_t^{-1}$. This shows that $X$ has the properties of a forward stochastic flow of $C^{1, \beta}$-\diffeomorphisms, and it is generated by \eqref{eq:sde_integral_form}.

    Next, let $(b^n)_{n \in \N}$ be a sequence of coefficients for \eqref{eq:sde_integral_form} converging to $b$ according to~\eqref{eq:convergent_coeff}, and let $X_{s, t}^n(x) \coloneqq \bigl((g_t^n)^{-1} \circ Y_{s, t}^n \circ g_s^n\bigr)(x)$. Then by the above discussion, $X^n$ is for all $n \in \N$ a stochastic flow of $C^{1, \beta}$-diffeomorphisms. To prove the convergences \mbox{\eqref{eq:forward_stability}--\eqref{eq:forward_gradient_stability}}, we first claim that
    \begin{equation} \label{eq:convergence_transfer}
        \begin{aligned}
            & \sup_{x \in \R^d} \sup_{s \in (0, T)} \E\biggl[\sup_{t \in (s, T)} \bigl|X_{s, t}^n(x) - X_{s, t}(x)\bigr|^p\biggr] \\
            & \quad \lesssim_p \norm{({g^n})^{-1} - g^{-1}}_{L^\infty_t C^0_x}^p + \norm{\nabla g^{-1}}_{L^\infty_t C^0_x}^p \sup_{x \in \R^d} \E\biggl[\sup_{t \in (0, T)} \bigl|Y_{t}^n(x) - Y_{t}(x)\bigr|^p\biggr]
        \end{aligned}
    \end{equation}
    for all $n \in \N$. Indeed, we have
    \begin{equation*}
        \begin{aligned}
            & \sup_{x \in \R^d} \sup_{s \in (0, T)} \E\biggl[\sup_{t \in (s, T)} \bigl|X_{s, t}^n(x) - X_{s, t}(x)\bigr|^p\biggr] = \sup_{x \in \R^d} \E\biggl[\sup_{t \in (0, T)} \bigl|X_{t}^n(x) - X_{t}(x)\bigr|^p\biggr] \\
            & \quad = \sup_{x \in \R^d} \E\biggl[\sup_{t \in (0, T)} \bigl|(g_t^{n})^{-1} \bigl(Y_{t}^n(x)\bigr) - (g_t)^{-1}\bigl(Y_{t}(x)\bigr) \bigr|^p\biggr],
        \end{aligned}
    \end{equation*}
    and \eqref{eq:convergence_transfer} follows by inserting
    \begin{equation*}
        \bigl|(g_t^n)^{-1}\bigl(Y_{t}^n(x)\bigr) - g_t^{-1}\bigl(Y_{t}(x)\bigr)\bigr|^p \lesssim_p \norm{({g_t^n})^{-1} - g_t^{-1}}_{C^0_x}^p + \norm{\nabla g_t^{-1}}_{C^0_x}^p |Y_{t}^n(x) - Y_{t}(x)|^p.
    \end{equation*}
    Since we know that $(g^n)^{-1} \to g^{-1}$ in $L^\infty_t C^{1, \alpha'}_x$ from the proof of Lemma \ref{lemma:zvonkin_flow}, this proves \eqref{eq:forward_stability}. Next, we claim that
    \begin{equation} \label{eq:convergence_transfer_derivative}
        \begin{aligned}
            & \sup_{x \in \R^d} \sup_{s \in (0, T)}\E\biggl[\sup_{t \in (s, T)} \bigl|\partial_i X_{s, t}^n(x) - \partial_i X_{s, t}(x)\bigr|^p\biggr] \\
            & \lesssim_{p} \norm{\nabla(g^n)^{-1} - \nabla g^{-1}}_{L^\infty_t C^0_x}^p \sup_{x \in \R^d} \E\biggl[\sup_{t \in (0, T)}\bigl|\partial_i Y_t^n(x)\bigr|^p\biggr] \\
            & + \norm{\nabla g^{-1}}_{L^\infty_t C^{0, \alpha'}_x}^p \sup_{x \in \R^d} \E \biggl[\sup_{t \in (0, T)}\bigl|Y_t^n(x) - Y_t(x)\bigr|^{2 \alpha' p}\biggr]^\frac{1}{2} \sup_{x \in \R^d} \E\biggl[\sup_{t \in (0, T)} \bigl|\partial_i Y_t^n(x)\bigr|^{2p}\biggr]^\frac{1}{2} \\
            & + \norm{\nabla g^{-1}}_{L^\infty_t C^0_x}^p \sup_{x \in \R^d} \E\biggl[\sup_{t \in (0, T)}\bigl|\partial_i Y_t^n(x) - \partial_i Y_t(x)\bigr|^p\biggr]
        \end{aligned}
    \end{equation}
    for all $n \in \N$ and $i = 1, ..., d$. To see this, observe that for the $i$-th partial derivative, we have
    \begin{equation*}
        \begin{aligned}
            & \sup_{x \in \R^d} \sup_{s \in (0, T)} \E\biggl[\sup_{t \in (s, T)} \bigl|\partial_i X_{s, t}^n(x) - \partial_i X_{s, t}(x)\bigr|^p\biggr] \\
            & \quad = \sup_{x \in \R^d} \E\biggl[\sup_{t \in (0, T)} \bigl|\partial_i X_{t}^n(x) - \partial_i X_{t}(x)\bigr|^p\biggr] \\
            & \quad = \sup_{x \in \R^d} \E\biggl[\sup_{t \in (0, T)} \bigl|\partial_i \bigl((g_t^n)^{-1} \circ Y_{t}^n\bigr)(x) - \partial_i \bigl((g_t)^{-1} \circ Y_{t}\bigr)(x)\bigr|^p\biggr].
        \end{aligned}
    \end{equation*}
    Using the estimate
    \begin{equation*}
        \begin{aligned}
            & \bigl|\partial_i \bigl((g_t^n)^{-1} \circ Y_{t}^n\bigr)(x) - \partial_i \bigl((g_t)^{-1} \circ Y_{t}\bigr)(x)\bigr|^p & \\
            & \lesssim_{p} \norm{\nabla(g_t^n)^{-1} - \nabla g_t^{-1}}_{C^0_x}^p |\partial_i Y_t^n(x)|^p  + \norm{\nabla g_t^{-1}}_{C^{0, \alpha'}_x}^p |Y_t^n(x) - Y_t(x)|^{\alpha' p} |\partial_i Y_t^n(x)|^p \\
            & \quad + \norm{\nabla g_t^{-1}}_{C^0_x}^p |\partial_i Y_t^n(x) - \partial_i Y_t(x)|^p
        \end{aligned}
    \end{equation*}
    and H\"older's inequality gives \eqref{eq:convergence_transfer_derivative}. This proves \eqref{eq:forward_gradient_stability} by virtue of \eqref{eq:flow_stability}, \eqref{eq:gradient_flow_stability} and \eqref{eq:gradient_flow_boundedness}, and concludes part (i) of the proof.

    For the proof of (ii) for the backward equation \eqref{eq:backward_sde}, we define
    \begin{equation*}
        \hat{X}_{s, t}(x) \coloneqq \bigl(\hat{g}_s^{-1} \circ \hat{Y}_{s, t} \circ \hat{g}_t\bigr)(x),
    \end{equation*}
    where $\hat{Y}$ is the backward stochastic flow of $C^{1, \beta}$-\diffeomorphisms generated by the solutions $\hat{Y}_s^{x, s}$ of \eqref{eq:zvonkin_process_backward_renamed}, provided by Theorem~\ref{thm:kunita_flow_theorem_backward}. It is evident that a lemma analogous to Lemma \ref{lemma:zvonkin_flow} holds. One can now show, in the same way as above, that $\hat{X}$ is indeed a backward stochastic flow of $C^{1, \beta}$-\diffeomorphisms generated by~\eqref{eq:backward_sde}. 
    
    Let us finally prove that the inverse of the forward flow, $X^{-1}$, coincides with the backward flow $\hat{X}$. Since $(X_{s, t}(x))_{t \in [s, T]}$ satisfies \eqref{eq:sde}, we have
    \begin{equation*}
        X_{s, t}(X_{s, t}^{-1}(x)) = X_{s, t}^{-1}(x) + \int_s^t b_r(X_{s, r}(X_{s, t}^{-1}(x)))\, dr + \int_s^t dW_r.
    \end{equation*}
    But $X^{-1}$ is the inverse of the forward flow $X$, and thus $X_{s, r}(X_{s, t}^{-1}(x)) = X_{r, t}^{-1}(x)$. This means that
    \begin{equation*}
            X_{s, t}^{-1}(x) = x - \int_s^t b_r(X_{r, t}^{-1}(x))\, dr - \int_s^t dW_r = x - \int_s^t b_r(X_{r, t}^{-1}(x))\, dr - \int_s^t \hat{d}W_r
    \end{equation*}
    holds $\P$-a.s., which is exactly the equation uniquely satisfied by paths of the backward flow $(\hat{X}_{s, t}(x))_{s \in [0, t]}$. We conclude that $\hat{X} = X^{-1}$ $\P$-a.s.
\end{proof}

\begin{remark}
    There is no technical need to take the H\"older exponent $\alpha'$ different from $\alpha$ in \eqref{eq:convergent_coeff}; we have only included this option to adopt the weakest assumptions necessary to ensure the validity of the proof.
\end{remark}

\begin{remark}
    The processes $(X_{s, t}(x))_{t \in [s, T]}$ of the forward flow $X_{s, t}(x)$ are clearly $C^{0, \alpha}$-\semimartingales (which can be seen directly from \eqref{eq:sde_integral_form}), but note that they are not necessarily $C^{1, \beta}$-\semimartingales, even though the flow is a $C^{1, \beta}$-\diffeomorphism.
\end{remark}

\begin{corollary} \label{cor:flow_uniform_convergence}
    Assume that $b$ satisfies \eqref{eq:h1}. Let $(b^n)_{n \in \N}$ be a convergent sequence of coefficients according to \eqref{eq:convergent_coeff} with corresponding stochastic flows $X^n$ and $X$. Then $\P$-a.s.~for all $(s, t) \in [0, T]^2$, there are subsequences of $X_{s, t}^n$ and $(X_{s, t}^n)^{-1}$ which converge to $X_{s, t}$ and $X_{s, t}^{-1}$ locally uniformly on $\R^d$. Moreover, $\P$-a.s.~for all $(s, t) \in [0, T]^2$ there are subsequences of $\nabla X_{s, t}^n$ and $\nabla (X_{s, t}^n)^{-1}$ which converge to $\nabla X_{s, t}$ and $\nabla X_{s, t}^{-1}$ locally uniformly on $\R^{d\times d}$.
\end{corollary}

\begin{proof}
    We show convergence only for $X^n$; the other convergences follow by analogous arguments. Using Fubini's theorem, we have
    \begin{equation*}
        \begin{aligned}
            \E\biggl[\int_{K} |X_{s, t}^n(x) - X_{s, t}(x)|^p\, dx \biggr]& \leq \E\biggl[\int_{K} \sup_{t \in (s, T)} |X_{s, t}^n(x) - X_{s, t}(x)|^p\, dx \biggr] \\
            & = 
            \int_K \E\biggl[\sup_{t \in (s, T)} |X_{s, t}^n(x) - X_{s, t}(x)|^p \biggr]\, dx
        \end{aligned}
    \end{equation*}
    for any compact set $K \subset \R^d$. Since the right-hand side converges by Theorem~\ref{thm:main_flow}, we get $\P$-a.s.~convergence in a subsequence of $X_{s, t}^n$ to $X_{s, t}$ in $L^p_{\loc}(\R^d)$. Taking a further subsequence, the integrand $|X_{s, t}^n(x) - X_{s, t}(x)|^p$ converges $\P$-a.s.~for a.e.~$x \in K$. But since all $X_{s, t}^n$ are uniformly $C^{1, \beta}$-regular, this yields convergence locally uniformly on $\R^d$.
\end{proof}

\section{Application to linear SPDEs} \label{sec:main_proofs}

This section contains proofs of Theorem~\ref{thm:main_transport} and Theorem~\ref{thm:main_continuity}. We begin by explaining the relevant solution concepts for the STE \eqref{eq:stoch_transport_strong} and the SCE \eqref{eq:stoch_continuity_strong}, and the duality relation between them.

$\BV_{\loc}$-solutions of the STE will be defined as follows.

\begin{definition} \label{def:stoch_transport_weak}
    Let $u_\initial \in \BV_{\loc}(\R^d)$. A $\BV_{\loc}$-solution of \eqref{eq:stoch_transport_strong} is a random field $u\in L^\infty(\R^d \times (0, T)\times \Omega)$ such that
    \begin{enumerate}[label=(\roman*)]
        \item $u_t(\cdot, \omega) \in \BV_\loc(\R^d)$ for a.e.~$(t, \omega) \in [0, T] \times \Omega$, i.e.
        \begin{equation} \label{eq:bounded_bv}
        \int_{\R^d} \vartheta(x)\, |\nabla u_t|(dx) < \infty
    \end{equation}
    for all $\vartheta \in C_c^\infty(\R^d)$, for a.e.~$t\in [0, T]$ $\Pas$.
    \item For all $\psi \in L^q((0, T); C_c(\R^d; \R^d))$, the process $t \mapsto \int_{\R^d} \psi(x, t) \cdot \nabla u_t(dx)$ is progressively measurable.
    \item The process $t \mapsto \int_{\R^d} \vartheta u_t\, dx \in L^\infty((0, T) \times \Omega)$ has a representative which is a continuous $\F_t$-\semimartingale and satisfies
    \begin{equation} \label{eq:stoch_transport_weak}
        \begin{aligned}
            \int_{\R^d} u_t \vartheta\, dx & = \int_{\R^d} u_\initial \vartheta\, dx - \int_0^t \int_{\R^d} \vartheta b_r \cdot \nabla u_r(dx) dr \\
            & \quad - \int_0^t \biggl(\int_{\R^d} \vartheta \nabla u_r(dx)\biggr) \circ dW_r
        \end{aligned}
    \end{equation}
    for all $t \in [0, T]$ and $\vartheta \in C_c^\infty(\R^d)$, $\Pas$.
    \end{enumerate}
\end{definition}

In \eqref{eq:stoch_transport_weak}, the derivative $\nabla u$ denotes the $d$-dimensional measure with components $\partial_i u \in \radon_{\loc}(\R^d)$ for $i = 1, ..., d$. In view of the regularity of $b$ given by \eqref{eq:h1}, there is no problem in writing down the product $b \cdot \nabla u$. Moreover, since
\begin{equation*}
    \int_{\R^d} \vartheta \partial_i u_t(dx) = - \int_{\R^d} \partial_i \vartheta u_t\, dx
\end{equation*}
and the right-hand side is a continuous $\F_t$-\semimartingale, the stochastic integral is well-defined. It can be written as an Itô integral by observing that
\begin{equation*}
    \int_0^t \biggl(\int_{\R^d} \vartheta \nabla u_r(dx)\biggr) \circ dW_r = \int_0^t \biggl(\int_{\R^d} \vartheta \nabla u_r(dx)\biggr) dW_r - \frac{1}{2} \int_0^t \int_{\R^d} \nabla \vartheta \cdot \nabla u_r(dx) dr,
\end{equation*}
which can bee seen by using $\partial_i \vartheta$, $i = 1, ..., d$, as test functions in \eqref{eq:stoch_transport_weak}.

Next, we define weak solutions for the SCE.

\begin{definition} \label{def:stoch_continuity_weak}
    Let $\mu_\initial \in \radon_{\loc}(\R^d)$. A weak solution of \eqref{eq:stoch_continuity_strong} is a random measure $\mu\colon \Omega \times [0, T] \times \B(\R^d) \to \R$ such that
    \begin{enumerate}[label=(\roman*)]
        \item $\mu_t(\omega, \cdot) \in \meas_\loc(\R^d)$ a.e.~$(t, \omega) \in [0, T] \times \Omega$, i.e.
        \begin{equation*} \label{eq:bounded_measure}
            \int_{\R^d} \vartheta(x)\, |\mu_t|(dx) < \infty,
        \end{equation*}
        for all $\vartheta \in C_c^\infty(\R^d)$, for a.e.~$t\in [0, T]$ $\Pas$.
        \item For all $\psi \in L^q((0, T); C_c(\R^d))$, the process $t \mapsto \int_{\R^d} \psi(x, t)\, \mu_t(dx)$ is progressively measurable.
        \item The process $t \mapsto \int_{\R^d} \vartheta\, d\mu_t$ is a continuous $\F_t$-\semimartingale and satisfies
        \begin{equation} \label{eq:stoch_continuity_weak}
            \begin{aligned}
                \int_{\R^d} \vartheta\, \mu_t(dx) & = \int_{\R^d} \vartheta\, \mu_\initial(dx) + \int_0^t \int_{\R^d} b_t \cdot \nabla \vartheta\, \mu_r(dx) dr \\
                & \quad + \int_0^t \biggl(\int_{\R^d} \nabla \vartheta\, \mu_r(dx) \biggr) \circ dW_r
            \end{aligned}
        \end{equation}
        for all $t \in [0, T]$ and $\vartheta \in C_c^\infty(\R^d)$, $\Pas$.
        \end{enumerate}
\end{definition}

Note that the product $b\mu$ in \eqref{eq:stoch_continuity_weak} is well-defined. As before, the SCE can be written in Itô form using the formula
\begin{equation*}
    \int_0^t \biggl(\int_{\R^d} \nabla \vartheta\, \mu_r(dx) \biggr) \circ dW_r = \int_0^t \biggl(\int_{\R^d} \nabla \vartheta\, \mu_r(dx) \biggr) dW_r + \int_0^t \int_{\R^d} \Delta \vartheta\, \mu_r(dx) dr,
\end{equation*}
which can be derived from \eqref{eq:stoch_continuity_weak} with test functions $\partial_i \vartheta$, $i = 1, ..., d$. 

As in the deterministic case, there is a natural duality between solutions of the STE and the SCE. To prove this, we will perform calculations on mollified versions of $\BV_\loc$-solutions of the STE and weak solutions of the SCE, and then pass to the limit. The next lemma will be useful for the latter.

\begin{lemma} \label{lemma:mollification_convergence}
    Assume that $f \in L^1((0, T); C^0_c(\R^d))$ and let $(\nu_t)_{t \in [0, T]}$ be a family of measures in $\meas_\loc(\R^d)$ with $\int_{0}^{T} \int_{\R^d} \vartheta(x)\, |\nu_t|(dx)dt < \infty$ for all $\vartheta \in C^0_c(\R^d)$. Let $f^\varepsilon = (f * \rho^\varepsilon)$ and $\nu^\varepsilon = (\nu * \rho^\varepsilon)$ for a family of standard mollifiers $(\rho^\varepsilon)_{\varepsilon > 0}$ in $C^{\infty}_c(\R^d)$. Then
    \begin{equation*}
        \lim_{\varepsilon \to 0} \int_0^t \int_{\R^d} f_r^\varepsilon(x) \nu_r^\varepsilon(x)\, dx dr = \int_0^t \int_{\R^d} f_r(x)\, \nu_r(dx) dr
    \end{equation*}
    for all $t \in [0, T]$.
\end{lemma}

The proof of the above lemma is relatively standard, combining the uniform convergence of $f^\varepsilon$ to $f$ and the locally weak-$*$ convergence of $\nu^\varepsilon$ to $\nu$ (see e.g.~\cite{ambrosio_fusco_pallara_2000} on local weak-$*$ convergence of measures). We are now able to prove the duality principle when one of the solutions of the STE and the SCE is sufficiently regular.

\begin{lemma} \label{lemma:duality}
    Assume that $b$ is a velocity field which satisfies \eqref{eq:h1}. Let $u$ be a $\BV_{\loc}$-solution of the STE \eqref{eq:stoch_transport_strong} with $u_\initial \in \BV_{\loc}(\R^d)$, and $\mu$ a weak solution of the SCE \eqref{eq:stoch_continuity_strong} with $\mu_\initial \in \radon_{\loc}(\R^d)$. Assume that $\P$-a.s., either
    \begin{equation*} \label{eq:duality_conds}
        \text{(i)}\quad u \in L^1((0, T); C^1_c(\R^d)) \qquad \text{or} \qquad (ii) \quad \mu \in L^1((0, T); C^0_c(\R^d)),
    \end{equation*}
    where the last condition should be understood as $\mu$ having a compactly supported, continuous density with respect to the Lebesgue measure on $\R^d$. Then
    \begin{equation} \label{eq:duality}
        \int_{\R^d} u_t(x) \mu_t(dx) = \int_{\R^d} u_\initial(x) \mu_\initial(dx), \qquad t \in [0, T]
    \end{equation}
    $\P$-a.s.
\end{lemma}

\begin{proof}
    Let $(\rho^\varepsilon)_{\varepsilon > 0}$ be a family of standard mollifiers in $C^{\infty}_c(\R^d)$. Since $u$ and $\mu$ are solutions of \eqref{eq:stoch_transport_weak} and \eqref{eq:stoch_continuity_weak}, the mollified functions ${u_t^\varepsilon(x) = (u_t * \rho^\varepsilon)(x)}$ and $\mu_t^\varepsilon = \int_{\R^d} \rho^\varepsilon(x-y)\, d\mu_t(y)$ are continuous $\F_t$-\semimartingales and satisfy
    \begin{equation*} \label{eq:mollified_equations}
        \begin{aligned}
            & u^\varepsilon_t = u_\initial^\varepsilon + \int_0^t (b_r \cdot \nabla u_r)^\varepsilon\, dr + \int_{0}^{t}\nabla u_r^\varepsilon\circ dW_r, \\
            & \mu_t^\varepsilon = \mu_\initial^\varepsilon + \int_0^t \nabla \cdot (b_r \mu_r)^\varepsilon\, dr + \int_0^t \nabla \mu_r^\varepsilon\circ dW_r
        \end{aligned}
    \end{equation*}
    for all $t \in [0, T]$, $\Pas$, where the other terms are defined similarly as convolutions against~$\rho^\varepsilon$. Itô's formula for the product gives
    \begin{equation*}
        \begin{aligned}
            u_t^\varepsilon \mu_t^\varepsilon & = u_\initial^\varepsilon \mu_\initial^\varepsilon + \int_0^t (b_r \cdot \nabla u_r)^\varepsilon \mu_r^\varepsilon \, dr + \int_0^t u_r^\varepsilon \nabla \cdot (b_r \mu_r)^\varepsilon\, dr \\
            & \quad + \int_{0}^{t}\nabla u_r^\varepsilon \mu_r^\varepsilon \circ dW_r + \int_0^t u_r^\varepsilon \nabla \mu_r^\varepsilon\circ dW_r.
        \end{aligned}
    \end{equation*}
    Assume first that condition (i) holds. Then $u_t^\varepsilon$ has compact support in $\R^d$ for a.e.~$t \in [0, T]$, and integrating the product $u^\varepsilon \mu^\varepsilon$ yields
    \begin{equation} \label{eq:mollified_duality}
        \begin{aligned}
            \int_{\R^d} u_t^\varepsilon \mu_t^\varepsilon\, dx & = \int_{\R^d} u_\initial^\varepsilon \mu_\initial^\varepsilon\, dx + \int_0^t \int_{\R^d} (b_r \cdot \nabla u_r)^\varepsilon \mu_r^\varepsilon\, dx dr \\
            & \quad - \int_0^t \int_{\R^d} \nabla u_r^\varepsilon \cdot (b_r \mu_r)^\varepsilon\, dx dr,
        \end{aligned}
    \end{equation}
    where we have changed the order of integration and then used integration by parts to cancel the stochastic integrals. Choose $t \in [0, T]$ such that $\mu_t \in \meas_{\loc}(\R^d)$ and $u_t \in C^1_c(\R^d)$ $\P$-a.s. Then $\mu^\varepsilon_t$ converges locally weak-$*$ to $\mu_t$, and $u_t^\varepsilon$ converges uniformly to $u_t$ on $\R^d$, so we get
    \begin{equation*}
        \lim_{\varepsilon \to 0} \int_{\R^d} u_t^\varepsilon(x) \mu_t^\varepsilon(x)\, dx = \int_{\R^d} u_t(x) \mu_t(dx), \qquad t \in [0, T]
    \end{equation*}
    $\P$-a.s. By Lemma \ref{lemma:mollification_convergence} the last two terms in \eqref{eq:mollified_duality} both converge $\P$-a.s.~to 
    \begin{equation*}
        \int_0^t \int_{\R^d} (b_r \cdot \nabla u_r)\, \mu_r(dx)dr,
    \end{equation*}
    which means that in the limit we are left with \eqref{eq:duality} for all $t \in [0, T]$, $\Pas$. 

    If condition (ii) holds instead of (i), a similar argument with $\nabla u$ taking the role of the measure yields the same conclusion.
\end{proof}

At this point, we have all the necessary ingredients to prove Theorem~\ref{thm:main_transport}.

\begin{proof}[Proof of Theorem~\ref{thm:main_transport}]
    We first claim that the function $u_t(x) = u_\initial(X_t^{-1}(x))$ is a solution of \eqref{eq:stoch_transport_strong} in the sense of Definition~\ref{def:stoch_transport_weak}, where $X$ is the stochastic flow of $C^{1, \beta}$-\diffeomorphisms from Theorem~\ref{thm:main_flow}. Note that by a change of variables,
    \begin{equation} \label{eq:transport_substitution}
        \int_{\R^d} \vartheta(x) u_t(x)\, dx = \int_{\R^d} \vartheta(X_t(x)) \det(\nabla X_t(x)) u_\initial(x)\, dx
    \end{equation}
    holds for all $\vartheta \in C_c(\R^d)$, $\P$-a.s., from which we infer that $t \mapsto \int_{\R^d} \vartheta(x) u_t(x)\, dx$ is a continuous $\F_t$-semimartingale. Moreover, from the identity
    \begin{equation} \label{eq:transport_derivative_substitution}
            \int_{\R^d} \vartheta(x)\, \partial_i u_t(dx) = \int_{\R^d} \vartheta(X_t(x)) \det(\nabla X_t(x)) \bigl(\partial_i X_t^{-1}\bigr)(X_t(x)) \cdot \nabla u_\initial(dx)
    \end{equation}
    (see \cite{flandoli_gubinelli_priola}) we obtain
    \begin{equation*}
        \begin{aligned}
            \int_{\R^d} \vartheta(x)\, |\partial_i u_t|(dx) & \leq \sup_{\substack{\psi \in C_c(\R^d)\\ |\psi| \leq 1}} \int_{\R^d} \vartheta(x) \psi(x)\, \partial_i u_t(dx) \\
            & \leq \int_{\R^d} |\vartheta(X_t(x))| \det(\nabla X_t(x)) \bigl|\bigl(\partial_i X_t^{-1}\bigr)(X_t(x))\bigr| |\nabla u_\initial|(dx)
        \end{aligned}
    \end{equation*}
    for $i = 1,..., d$, which proves \eqref{eq:bounded_bv}.

    We show that $u_\initial(X_t^{-1}(x))$ satisfies equation \eqref{def:stoch_transport_weak}. Let $u_\initial^\delta  = (u_\initial * \rho^\delta)$ for a family of standard mollifiers $(\rho^\delta)_{\delta > 0}$, and let furthermore $X^\varepsilon$ be the flow generated by the mollified drift $b^\varepsilon = (b * \rho^\varepsilon)$, where $(\rho^\varepsilon)_{\varepsilon > 0}$ is another family of standard mollifiers. Then a classical result (see \cite[Chap.~4.4]{kunita_1990}) is that the function ${u^{\delta, \varepsilon} = u_\initial^\delta \bigl((X_t^{\varepsilon})^{-1}(x)\bigr)}$ satisfies
    \begin{equation} \label{eq:approximate_ste}
        \int_{\R^d} \vartheta u^{\delta, \varepsilon}_t\, dx = \int_{\R^d} \vartheta u^{\delta}_0\, dx + \int_0^t \int_{\R^d} \vartheta b_r^\varepsilon \cdot \nabla u_r^{\delta, \varepsilon}\, dx dr + \int_0^t \biggl(\int_{\R^d} \vartheta \nabla u^{\delta, \varepsilon}_r\, dx \biggr)\circ dW_r
    \end{equation} 
    for all $\vartheta \in C^\infty_c(\R^d)$ and $t \in [0, T]$, $\P$-a.s. In short, this can be obtained using the forward formula
    \begin{equation*} \label{eq:backward_flow_forward_formula}
        (X_{t}^\varepsilon)^{-1}(x) = x - \int_0^t \nabla (X_{r}^\varepsilon)^{-1}(x) b_r^\varepsilon(x)\, dr - \int_0^t \nabla (X_{r}^\varepsilon)^{-1}(x) \circ dW_r
    \end{equation*}
    for the backward flow $(X_t^\varepsilon)^{-1}$, combined with Ito's formula for the composition $u_\initial \circ (X_t^{\varepsilon})^{-1}$. The identities \eqref{eq:transport_substitution} and \eqref{eq:transport_derivative_substitution} imply that we can pass $\delta \to 0$ in \eqref{eq:approximate_ste} and get
    \begin{equation} \label{eq:approximate_ste_2}
        \int_{\R^d} \vartheta u^{\varepsilon}_t\, dx = \int_{\R^d} \vartheta u_\initial\, dx + \int_0^t \int_{\R^d} \vartheta b_r^\varepsilon \cdot \nabla u_r^{\varepsilon}\, dx dr + \int_0^t \biggl(\int_{\R^d} \vartheta \nabla u^{\varepsilon}_r\, dx \biggr)\circ dW_r,
    \end{equation}
    due to convergence of $u_\initial^\delta$ to $u_\initial$ for a.e.~$x \in \R^d$ and local weak-$*$ convergence of $\nabla u_\initial^\delta$ to $\nabla u_\initial$. It remains to show that each term in \eqref{eq:approximate_ste_2} converges as $\varepsilon \to 0$. Since for all $t \in [0, T]$, the mollified inverse flow $(X_t^\varepsilon)^{-1}$ converges $\P$-a.s.~locally uniformly in a subsequence to $X_t^{-1}$ by Corollary~\ref{cor:flow_uniform_convergence}, we have
    \begin{equation*}
        \begin{aligned}
            \lim_{\varepsilon_n \to 0} \int_{\R^d} \vartheta u^{\varepsilon_n}_t\, dx & = \lim_{\varepsilon_n \to 0} \int_{\R^d} \vartheta(x) u_\initial((X_t^{\varepsilon_n})^{-1}(x))\, dx \\
            & = \int_{\R^d} \vartheta(x) u_\initial(X_t^{-1}(x))\, dx = \int_{\R^d} \vartheta u_t\, dx
        \end{aligned}
    \end{equation*}
    for all $t \in [0, T]$, $\P$-a.s., by dominated convergence. Moreover, we have that for all $t \in [0, T]$, $\nabla u_t^\varepsilon$ converges in a subsequence $\P$-a.s.~locally weak-$*$ to $\nabla u_t$ on $\R^d$. Indeed,
    \begin{equation*}
        \begin{aligned}
            & \lim_{\varepsilon_n \to 0} \int_{\vartheta} \vartheta(x)\, \partial_i u_t^{\varepsilon_n}(dx) \\
            & \quad = \lim_{\varepsilon_n} \int_{\R^d} \vartheta(X_t^{\varepsilon_n}(x)) \det(\nabla X_t^{\varepsilon_n}(x)) \bigl(\partial_i (X_t^{\varepsilon_n})^{-1}\bigr)(X_t^{\varepsilon_n})\cdot \nabla u_\initial(dx)\\
            & \quad = \int_{\R^d} \vartheta(X_t(x)) \det(\nabla X_t(x)) \bigl(\partial_i X_t^{-1}\bigr)(X_t)\cdot \nabla u_\initial(dx) = \int_{\vartheta} \vartheta(x)\, \partial_i u_t(dx)
        \end{aligned}
    \end{equation*}
    for all $\vartheta \in C^\infty_c(\R^d)$ and $i = 1, ..., d$, also by Corollary~\ref{cor:flow_uniform_convergence}. We infer that the last two terms of \eqref{eq:approximate_ste_2} converge along a subsequence, so that we are left with \eqref{def:stoch_transport_weak}.

    Suppose that $u$ is any $\BV_\loc$-solution of \eqref{eq:stoch_transport_strong} for given $u_\initial \in \BV_\loc(\R^d)$. $\P$-a.s.~for all $\vartheta \in C_c(\R^d)$, we then have
    \begin{equation*}
        \int_{\R^d} \vartheta(x) u_t(X_t(x))\, dx = \int_{\R^d} u_t(x)\, (X_t)_\# \vartheta(dx) = \int_{\R^d} \vartheta(x) u_\initial(x)\, dx
    \end{equation*}
    for all $t \in [0, T]$ by Lemma \ref{lemma:duality}, since $(X_t)_\# \vartheta$ is a weak solution of the SCE \eqref{eq:stoch_continuity_weak} which $\P$-a.s.~belongs to $L^1((0, T); C^0_c(\R^d))$. This implies that $u_t(x) = u_\initial(X_t^{-1}(x))$ for a.e.~$x \in \R^d$ and all $t \in [0, T]$, $\P$-a.s., and thereby uniqueness.
\end{proof}

\begin{remark}
    The fact that $u_t(x) = u_\initial(X_t^{-1}(x))$ is a $\BV_\loc$-solution to the STE can be verified without smoothing the initial data and the flow, using \eqref{eq:transport_substitution} and the formulas for $X_t(x)$ and $\nabla X_t(x)$. However, this requires more calculation, and is not necessary when stability of the backward flow with respect to smoothing of the drift is known.
\end{remark}

We finally prove Theorem~\ref{thm:main_continuity}, where uniqueness will be defined as
\begin{equation} \label{eq:uniqueness_continuity}
    \P\biggl(\int_{\R^d} \vartheta\, d\mu_t = \int_{\R^d} \vartheta\ d\nu_t\ \text{for all}\ \vartheta \in C^\infty_c(\R^d)\ \text{and all}\ t \in [0, T]\biggr) = 1
\end{equation}
for any two weak solutions $\mu$ and $\nu$.

\begin{proof}[Proof of Theorem~\ref{thm:main_continuity}]
    We claim that the measure $\mu_t = (X_t)_\# \mu_\initial$ is a weak solution of the SCE \eqref{eq:stoch_continuity_strong}. Since
    \begin{equation} \label{eq:pushforward}
        \int_{\R^d} \vartheta(x)\, \mu_t(dx) = \int_{\R^d} \vartheta(X_t(x))\, \mu_\initial(dx),
    \end{equation}
    the process $t \mapsto \int_{\R^d} \vartheta\, d\mu_t$ is a continuous $\F_t$-\semimartingale, and applying Itô's formula on the right-hand side of \eqref{eq:pushforward}, we see that $\mu_t$ indeed satisfies \eqref{eq:stoch_continuity_weak}. Properties (i) and (ii) of Definition \ref{def:stoch_continuity_weak} can be checked using the pushforward formula for the solution.

    Let now $\mu$ be a any weak solution of \eqref{eq:stoch_continuity_strong}. Then
    \begin{equation*}
        \int_{\R^d} \vartheta(x)\, (X_t^{-1})_{\#} \mu_t (dx) = \int_{\R^d} \vartheta(X_t^{-1})\, \mu_t (dx) = \int_{\R^d} \vartheta(x)\, \mu_\initial(dx)
    \end{equation*}
    by Lemma \ref{lemma:duality}, since $\vartheta(X_t^{-1}(x))$ is a $\BV_\loc$-solution of the STE which is also $\P$-a.s.~in $L^1((0, T); C^1_c(\R^d))$. This proves the representation formula \eqref{eq:continuity_formula}, and thus uniqueness in the sense of \eqref{eq:uniqueness_continuity}.
\end{proof}

\appendix

\section{A Gronwall-type inequality} \label{appendix:gronwall}

When proving a priori regularity for solutions of the parabolic PDE \eqref{eq:model_eq_classical}, we would like to use a quantitative Gronwall-type inequality for time-dependent functions $u$ which satisfies an integral inequality
\begin{equation} \label{eq:generic_integral_inequality_1}
    u(t) \leq f(t) + \int_0^t g(t-s) u(s)\, ds
\end{equation}
on $(0, T)$, where $f$ and $g$ are nonnegative functions. In \cite[Lemma 3.1]{fedrizzi_flandoli_2013} it is claimed that \eqref{eq:generic_integral_inequality_1} implies
\begin{equation} \label{eq:wrong_inequality}
    u(t) \leq f(t) + \int_0^t f(s) g(t-s) \exp\biggl(\int_s^t g(t-r)\, dr \biggr)\, ds
\end{equation}
on $(0, T)$, given that $g \in L^1(0, T)$ and $g u \in L^1(0, T)$. (Their result is written backward in time, but can easily be converted to the above by a change of variables). This claim is not true in general, as the following example shows: If
\begin{equation*}
    u(t) = 1+t, \qquad f(t) = 1,\qquad g(t) = e^{-t},
\end{equation*}
then $u(t)$ satisfies \eqref{eq:generic_integral_inequality_1} with equality, but the right-hand side of \eqref{eq:wrong_inequality} is
\begin{equation*}
    1 + \int_0^t e^{-(t-s)} \exp\biggl(\int_s^t e^{-(t-r)}\, dr\biggr)\, ds = e^{1-e^{-t}} < 1 + t
\end{equation*}
for all $t > 0$. The inequality \eqref{eq:wrong_inequality} would be valid if one in addition assumed $g$ to be nondecreasing, but this is not the case for terms of the form $1/t^\beta$ typically appearing in heat kernel estimates.

To correct this situation, we prove a slightly different, more general result.

\begin{proposition} \label{prop:gronwall}
    Assume that $u \in L^\infty(0, T)$ satisfies the integral inequality
    \begin{equation} \label{eq:generic_integral_inequality_2}
        u(t) \leq f(t) + \int_0^t g(t, s) h(s) u(s)\, ds
    \end{equation}
    for a.e.~$t \in (0, T)$, where $f, g, h$ are nonnegative functions such that $f \in L^\infty(0, T)$,
    \begin{equation*}
        \sup_{t \in (0, T)} \norm{g(t, \cdot)}_{L^p(0, t)} < \infty
    \end{equation*}
    and $h \in L^q(0, T)$, for $\nicefrac{1}{p} + \nicefrac{1}{q} = 1$ with $p, q \in (1, \infty)$. Then for a.e.~$t \in (0, T)$, there is a constant
    \begin{equation*}
        \gronwconst = \gronwconst\biggl(\sup_{s \in (0, t)} \norm{g(s, \cdot)}_{L^p(0, s)}, \norm{h}_{L^q(0, t)}\biggr) > 0
    \end{equation*}
    which is continuous and nondecreasing in the two arguments, such that
    \begin{equation} \label{eq:modified_gronwall}
        u(t) \leq \gronwconst \norm{f}_{L^\infty(0, t)}.
    \end{equation}
\end{proposition}

\begin{proof}
    Define the operator
    \begin{equation*}
        \mathcal{I}[f](t) \coloneqq \int_0^t g(t, s) h(s) f(s)\, ds,
    \end{equation*}
    which is clearly linear and nondecreasing in the first argument. Then \eqref{eq:generic_integral_inequality_2} can be written as $u(t) \leq f(t) + \mathcal{I}[u](t)$. Iterating this inequality $N$ times yields 
    \begin{equation} \label{eq:gronwall_expansion}
        u(t) \leq \sum_{n = 0}^N \mathcal{I}^n[f](t) + \mathcal{I}^{N+1}[u](t),
    \end{equation}
    where $\mathcal{I}^0[f] = f$ and $\mathcal{I}^n[f] = \underbrace{\mathcal{I} \circ ... \circ \mathcal{I}}_{n}[f]$. We claim that
    \begin{equation} \label{eq:iterated_kernel}
        \mathcal{I}^n[f](t) \leq \biggl(\frac{1}{(n!)^\frac{1}{q}} \sup_{s \in (0, t)} \norm{g(s, \cdot)}_{L^p(0, s)}^{n} \norm{h}_{L^q(0, t)}^{n}\biggr) \norm{f}_{L^\infty(0, t)}
    \end{equation}
    for all $n \geq 0$. If this is indeed the case, then we can pass $N \to \infty$ in \eqref{eq:gronwall_expansion}, and obtain the bound
    \begin{equation} \label{eq:gronwconst}
        u(t) \leq \underbrace{\biggl(\sum_{n = 0}^\infty \frac{1}{(n!)^\frac{1}{q}} \sup_{s \in (0, t)} \norm{g(s, \cdot)}_{L^p(0, s)}^{n} \norm{h}_{L^q(0, t)}^{n}\biggr)}_{\gronwconst} \norm{f}_{L^\infty(0, t)}.
    \end{equation}
    We prove \eqref{eq:iterated_kernel} by induction. Using H\"older's inequality, $\mathcal{I}^1[f]$ is clearly bounded by
    \begin{equation*}
        \begin{aligned}
            \mathcal{I}^1[f](t) & \leq \norm{g(t, \cdot)}_{L^p(0, t)} \norm{h}_{L^q(0, t)} \norm{f}_{L^\infty(0, t)} \\
            & \leq \sup_{s \in (0, t)} \norm{g(s, \cdot)}_{L^p(0, s)} \norm{h}_{L^q(0, t)} \norm{f}_{L^\infty(0, t)}.
        \end{aligned}
    \end{equation*}
    Assuming that \eqref{eq:iterated_kernel} holds, we have
    \begin{equation*} \label{eq:induction_step}
        \begin{aligned}
            \mathcal{I}_{n+1}(t) & \leq \frac{1}{(n!)^\frac{1}{q}} \biggl(\int_0^t g(t, s) h(s) \sup_{r \in (0, s)}\norm{g(r, \cdot)}_{L^p(0, r)}^n \norm{h}_{L^q(0, s)}^n\, ds\biggr) \norm{f}_{L^\infty(0, t)} \\
            & \leq \frac{1}{(n!)^\frac{1}{q}} \biggl(\int_0^t g^p(t, s) \sup_{r \in (0, s)}\norm{g(r, \cdot)}_{L^p(0, r)}^{np} \, ds\biggr)^\frac{1}{p} \\
            & \quad \times \biggl(\int_0^t h^q(s) \norm{h}_{L^q(0, s)}^{nq}\, ds \biggr)^\frac{1}{q} \norm{f}_{L^\infty(0, t)}
        \end{aligned}
    \end{equation*}
    For the second integral, we use the identity
    \begin{equation*}
        \begin{aligned}
            & \int_0^t h^q(s) \norm{h}_{L^q(0, s)}^{nq}\, ds = \int_0^t h^q(s) \biggl(\int_0^s h^q(r)\, dr\biggr)^n\, ds \\
            & = \int_0^t \int_0^s \cdots \int_0^s \bigl(h^q(s) h^q(r_1)...h^q(r_n) \bigr)\, dr_1...dr_n ds = \frac{1}{n+1} \biggl(\int_0^t h^q(s)\, ds\biggr)^{n+1},
        \end{aligned}
    \end{equation*}
    due to the integrand $h^q \cdot ... \cdot h^q$ being symmetric over the $(n+1)$-dimensional hyperplane $\{s = r_1 = ... = r_n\}$. For the first integral, we have
    \begin{equation*}
        \begin{aligned}
            \int_0^t g^p(t, s) \sup_{r \in (0, s)}\norm{g(r, \cdot)}_{L^p(0, r)}^{np} \, ds & \leq \sup_{s \in (0, t)} \norm{g(s, \cdot)}_{L^p(0, s)}^{np} \int_0^t g^p(t, s)\, ds \\
            & \leq \sup_{s \in (0, t)} \norm{g(s, \cdot)}_{L^p(0, s)}^{(n+1)p}.
        \end{aligned}
    \end{equation*}
    This concludes the induction step and thereby \eqref{eq:iterated_kernel}.
\end{proof}

For the special case $g(t, s) = g(t-s)$, the result can be simplified somewhat: Since in this case $\sup_{s \in (0, t)} \norm{g(s-\cdot)}_{L^p(0, s)} = \norm{g}_{L^p(0, t)}$, we obtain the following corollary.

\begin{corollary}
    Assume that $g(t, s) = g(t-s)$. Then \eqref{eq:modified_gronwall} holds with a constant $E = E(\norm{g}_{L^p(0, t)}, \norm{h}_{L^q(0, t)})$.
\end{corollary}

\bibliography{bibliography}
\bibliographystyle{abbrv}
\end{document}